\documentclass{amsart}
\usepackage{amsmath}
\usepackage{amssymb}
\usepackage[dvipdfmx]{graphicx}
\usepackage{amscd}
\usepackage{amsthm}
\usepackage{fullpage}
\usepackage{mathrsfs}
\usepackage{setspace}
\usepackage{mathtools}
\usepackage{xcolor}
\usepackage[all]{xy}
\usepackage{url}
\usepackage{bbm}
\usepackage[dvipdfmx, colorlinks=true, pagebackref]{hyperref}
\usepackage[lite,abbrev,alphabetic]{amsrefs}

\usepackage{autobreak}

\theoremstyle{definition}
\newtheorem{dfn}{Definition}
\newtheorem{thm}[dfn]{Theorem}
\newtheorem{lem}[dfn]{Lemma}
\newtheorem{prop}[dfn]{Proposition}
\newtheorem{rem}[dfn]{Remark}

\numberwithin{equation}{section}

\def\N{\mathbb{N}}
\def\C{\mathbb{C}}
\def\R{\mathbb{R}}
\def\Q{\mathbb{Q}}
\def\Z{\mathbb{Z}}
\def\A{\mathbb{A}}

\def\GL{\mathrm{GL}}

\def\SL{\mathrm{SL}}

\def\SO{\mathrm{SO}}

\def\t{\,{}^t}
\def\d{\,\mathrm{d}}

\def\bsl{\backslash}
\def\inf{\infty}

\def\fd{{\mathfrak{d}}}

\def\fo{\mathfrak{o}}

\def\fD{{\mathfrak{D}}}

\def\fT{{\mathfrak{T}}}

\newcommand{\cA}{\mathcal{A}}

\newcommand{\cF}{\mathcal{F}}

\newcommand{\cH}{\mathcal{H}}

\newcommand{\cL}{\mathcal{L}}

\newcommand{\cN}{\mathcal{N}}
\newcommand{\cO}{\mathcal{O}}

\newcommand{\cT}{\mathcal{T}}

\DeclareMathOperator{\Gal}{Gal}

\DeclareMathOperator{\Nm}{Nm}
\DeclareMathOperator{\Tr}{Tr}

\def\bb{\mathbf{b}}
\def\ba{\mathbf{a}}

\def\sF{\mathscr{F}}

\def\Re{\mathop{\mathrm{Re}}}
\def\Im{\mathop{\mathrm{Im}}}
\def\Tr{\mathop{\mathrm{Tr}}}
\def\Nm{\mathop{\mathrm{Nm}}}

\def\Gal{\mathop{\mathrm{Gal}}}

\def\new{\mathrm{new}}

\def\eps{\varepsilon}

\def\ul{\underline{l}}
\def\fF{\mathfrak{F}}

\def\rv{\Q^{(3)}}

\newcommand{\HG}[1]{\mathcal{H}_{#1}^{\Gamma}}

\newcommand{\eH}[2]{\cH_{[#1,#2]}}
\newcommand{\modF}[2]{F_{[#1,#2]}}

\begin{document}

\title{Computations of algebraic modular forms associated with the definite quaternion algebra of discriminant $2$}

\author{Hiroyuki Ochiai}
\author{Satoshi Wakatsuki}
\author{Shun'ichi Yokoyama}

\address{Hiroyuki Ochiai \\
Institute of Mathematics for Industry, Kyushu University\\
744, Motooka, Nishi-ku, Fukuoka, 819-0395, Japan}
\email{ochiai@imi.kyushu-u.ac.jp}

\address{Satoshi Wakatsuki \\
Faculty of Mathematics and Physics, Institute of Science and Engineering\\
Kanazawa University\\
Kakumamachi, Kanazawa, Ishikawa, 920-1192, Japan}
\email{wakatsuk@staff.kanazawa-u.ac.jp}

\address{Shun'ichi Yokoyama \\
Department of Mathematical Sciences\\
Graduate School of Science, Tokyo Metropolitan University\\
1-1 Minami-Osawa, Hachioji-shi, Tokyo, 192-0397, Japan}
\email{s-yokoyama@tmu.ac.jp}

\date{\today}

\begin{abstract}
In this paper, we present an algorithm to compute a basis of the space of algebraic modular forms on the maximal order of the definite quaternion algebra of discriminant $2$, and provide a database of such bases. 
A main application of our database is to obtain congruence relations of algebraic modular forms, which lead non-vanishing theorems for prime twists of modular $L$-functions. 
\end{abstract}

\maketitle

\setcounter{tocdepth}{1}

\tableofcontents


\section{Introduction}

The importance of algebraic modular forms on definite quaternion algebras lies in their applications to the study of holomorphic modular forms via theta series. The two main applications are as follows. 
    \begin{itemize}
        \item[(A)] Modular forms of integral weights and Basis problem. See e.g. \cites{Eic73, Shi72}. 
        \item[(B)] Modular forms of half-integral weights and Toric periods. See e.g. \cites{Wal81, Wal85, Wal91}.
    \end{itemize}
The motivation of this paper lies in (B). 
Usually, when we talk about the computation of algebraic modular forms on definite quaternion algebras, we often assume the computation of Brandt matrices in connection with (A), see e.g. \cite{Piz80}. 
In fact, the Brandt matrix algorithm is well established, and the function to compute the Brandt matrix is incorporated in \textsf{Magma} \cite{BCP97}. 
On the other hand, as far as (B) is concerned, it seems that the algorithm for computing the toric periods and the theta series for algebraic modular forms has not been established. 


The relationship between quaternion algebras and modular forms has been clarified in various ways, including what I mentioned above, but apart from that, Ihara \cite{Iha64} and Ibukiyama \cite{Ibu84} developed algebraic modular forms for algebraic groups of higher dimension.
After those, the general notion of an algebraic modular form was formulated by Gross \cite{Gro99}. 
If we restrict our attention to the special case (the class number $1$) in his notion, 
for a compact Lie group $G$ and a discrete subgroup $\Gamma$ of $G$, a space of algebraic modular forms is defined by the $\Gamma$-invariant subspace of a representation space of a finite dimensional representation of $G$. 
Hence, algebraic modular forms on definite quaternion algebras are defined as vectors in a representation space of $\mathrm{SU}(2)$. 
Therefore, it is important to choose a suitable realization of the representation space according to the purpose. 
In the research related to (B), it is preferable to treat the representation of $\mathrm{SO}(3)\cong \mathrm{SU}(2)/\{\pm 1\}$ from the viewpoint of theta correspondence. 
In addition, for the application of our motivation to toric periods, which is discussed below, it is necessary to choose the space of homogeneous harmonic polynomials of $3$ variables as a realization of the representation space of $\SO(3)$, and explicitly describe algebraic modular forms as Ihara \cite{Iha64} and Ibukiyama \cite{Ibu84} did. 
In this paper, we treat the Hurwitz quaternion order of the definite quaternion algebra of discriminant $2$, and present an algorithm for computing bases of spaces of algebraic modular forms as harmonic polynomials. 
Naturally, our algorithm can be applied to other orders of class number $1$ as well, and we plan to generalize it to the case where the class number is greater than $1$ (we already treated such the case in \cite{SWY22} for computations of toric periods).

Our main motivation is the study of non-vanishing theorems of central values of prime twists of modular $L$-functions in \cite{CW23}. 
In two examples, Chida and the second author proved that, for every quadratic field with a prime discriminant, the corresponding toric period does not vanish, that is, we obtain the non-vanishing of the central value of the corresponding $L$-function if its root number is positive. 
These are interesting examples for the conjecture of \cite{CKRS02}, and we would like to study the behavior of the central values for prime twists in more examples to refine their conjecture (see \cite{CW23} for the detail). 
The key point of the proof for the non-vanishing of the toric period in \cite{CW23} is to define the integral structure of algebraic modular forms along the lattice introduced in \cite{Gro87} and to consider congruence relations of values of algebraic modular forms at the lattice points. 
Note that the explicit description of algebraic modular forms as harmonic polynomials is essential to obtain the congruence relations of algebraic modular forms. 
In this paper, we show that a lot of algebraic modular forms obtained by running our algorithm with \textsf{Magma} \cite{BCP97} satisfy similar congruence relations, and hence we obtain many examples of the non-vanishing theorem for prime twists of modular $L$-functions (cf. \S\ref{sec:4}). 
This fact suggests that the congruence relation satisfied by the algebraic modular forms is not coincidence, and we expect to see some implications in the future.

The obtained algorithm gives a database of bases of spaces of algebraic modular forms. 
See Appendices~\ref{sec:app}~and~\ref{sec:app2} and our website \cite{OWY24} for the database. 
As a new fact that we are able to observe from the database, we discuss a multiplicative property of algebraic modular forms in \S\ref{sec:fac} and \S\ref{sec:refine}. 
This remarkable property may expect a connection to the study of the Waldspurger lifts of algebraic modular forms of odd degrees. 
It is interesting to study the explicit relations between the Fourier coefficients of the Waldspurger lifts of algebraic modular forms and the central values of the $L$-functions, see \cites{Gro87, ST18}. 
Since the study for odd degrees has not been established so far, we plan to study the Waldspurger lifts by using the multiplicative property, and conduct concrete numerical experiments based on the database in the future.

This paper is organized as follows.
In \S \ref{sec:1}, we review the basic facts of the theory of algebraic modular forms on quaternion algebra of discriminant $2$.  
In \S \ref{sec:2}, we prove dimension formulas for the spaces of algebraic modular forms. 
They are known and unnecessary for our algorithm in \S\ref{sec:3}, but the discussion allows us to know what the relationship is between the dimension and the information of the integer ring. 
The dimension formulas are used for the factorization problems in \S\ref{sec:fac}.
In \S \ref{sec:3}, we give an algorithm for computing a basis of the space of algebraic modular forms. This is the main result of this paper.
In \S \ref{sec:4}, based on the explicit expressions of algebraic modular forms obtained from the algorithm given in \S \ref{sec:3}, we give many examples in which we have congruence relations of algebraic modular forms at CM points and the non-vanishing of the central values of prime twists of $L$-functions. 
In \S \ref{sec:fac}, we discuss some factorization problems using our computational results. 
In \S \ref{sec:refine}, we discuss a simpler expression by using elementary symmetric polynomials.
In Appendix \ref{sec:app}, we describe basis obtained from our algorithm for degrees below $13$.
In Appendix \ref{sec:app2}, we list the simple expressions for the algebraic modular forms in Appendix \ref{sec:app}.

\medskip

\noindent
\textbf{Acknowledgments.} 
The authors thank Masataka Chida for his advice on congruence relations between algebraic modular forms and norms at CM points and allowing us to write it in this paper. 
The first author is partially supported by JSPS 
Grant-in-Aid for Transformative Research Areas (A) No.22H05107.
The second author is partially supported by JSPS Grant-in-Aid for Scientific Research (C) No.20K03565, (B) No.21H00972. 
The third author is supported by JSPS Grant-in-Aid for Scientific Research (C) No.20K03537.

\section{Setup}\label{sec:1}

\subsection{Quaternion algebra \texorpdfstring{$D$}{TEXT}}

A quaternion algebra $D$ over $\Q$ is defined by
\[
D=\Q+\Q\, i+\Q\, j+\Q\, ij, \quad \quad i^2=j^2=-1, \quad ij=-ji.
\]
This algebra $D$ is written as $\left(\frac{-1,-1}{\Q}\right)$, and its discriminant is $2$, see \cite{Voi21}. 
A $\Q$-involution $\iota$ on $D$ is defined as
\[
\iota(a_0+a_1i+a_2j+a_3ij)=a_0-a_1i-a_2j-a_3ij \quad (a_0,a_1,a_2,a_3\in \Q).
\]
We also set
\[
\Tr(x)=x+\iota(x),\quad \Nm(x)=x\, \iota(x) \quad (x\in D).
\]
It is known that $D$ is division, and the inverse of $x(\neq 0)$ is $\Nm(x)^{-1} \iota(x)$. 
The $\R$-algebra $D\otimes_\Q \R$ is just the Hamiltonian's quaternion algebra.

\subsection{Maximal order \texorpdfstring{$\cO$}{TEXT}}

A maximal order $\cO$ in $D$ is given by
\[
\cO=\Z+\Z i + \Z j + \Z w , \quad w=\frac{1+i+j+ij}{2} .
\]
The elements in $\cO$ are called the Hurwitz integral quaternions, cf. \cite{CS03}*{\S5}, and this lattice $\cO$ is known to be the $D_4$ root lattice for the bilinear form $\langle x,y\rangle=\frac12(\Nm(x+y)-\Nm(x)-\Nm(y))$. 
The group $\cO^\times$ is generated as 
\[
\cO^\times = \langle i,\;\; j, \;\; w  \rangle.
\]
Note that we have $i^2=j^2=w^3=-1$ and $\#\cO^\times=24$.

\subsection{Special orthogonal group \texorpdfstring{$\SO(Q)$}{SO(Q)}}\label{sec:Q}

A vector subspace $V_D$ of $D$ over $\Q$ is defined by 
\[
V_D\coloneqq \{ x\in D \mid \Tr(x)=0\}.  
\]
Take a basis $\bb_1$, $\bb_2$, $\bb_3$ of $V_D$ as
\[
\bb_1=-i+j+ij, \qquad \bb_2=i-j+ij, \qquad \bb_3=i+j-ij. 
\]
Let $\rv$ denote the vector space consisting of row vectors of degree $3$ over $\Q$.  
We identify $V_D$ with $\rv$ by
\[
\cT(a_1\bb_1+a_2\bb_2+a_3\bb_3)=\begin{pmatrix} a_1  & a_2 & a_3 \end{pmatrix} \in \rv.
\]
Set
\[
Q\coloneqq\begin{pmatrix} 3&-1&-1 \\ -1&3&-1 \\ -1&-1&3 \end{pmatrix} , \quad \text{and then} \quad  Q^{-1}=\frac{1}{4}\begin{pmatrix} 2&1&1 \\ 1&2&1 \\ 1&1&2 \end{pmatrix} . 
\]
When $x\in V_D$ and $\ba=\cT(x)\in \rv$, we have $\Nm(x)=\ba Q \t\ba$.

The special orthogonal group $\SO(Q)$ is defined as $\SO(Q)\coloneqq \{g\in \SL_3(\Q) \mid g Q \t g=Q\}$. 
The group $D^\times$ acts on $V_D$ as $x\cdot g =g^{-1}xg$ $(x\in V_D$, $g\in G)$. 
Then, we get a homomorphism $\cF : D^\times \to \SO(Q)$ such that
\[
\cT(x\cdot g)=\cT(x)\, \cF(g) \quad (x\in V_D, \;\; g\in D^\times). 
\]
By $\mathrm{Ker}(\cF)=\Q^\times$, we obtain the isomorphism $\Q^\times\bsl D^\times\cong \SO(Q)$.


\subsection{Harmonic polynomials}

For the variables $x_1$, $x_2$, $x_3$, we set
\[
\cA_l \coloneqq \{ f(x_1,x_2,x_3)\in\Q[x_1,x_2,x_3] \mid \text{$\deg f=l$ and $f$ is homogeneous}   \}.
\]
Note that $\dim \cA_l =\frac{(l+2)(l+1)}{2}$.
Each polynomial $f(x_1,x_2,x_3)$ in $\cA_l$ is identified with the polynomial function $V_D\ni a \mapsto f\circ\cT(a) \in \Q$. 
The Laplace operator $\Delta_Q$ for $Q$ is defined by $\Delta_Q= \sum_{1\le i,j\le 3} q'_{ij} \frac{\partial^2}{\partial x_i \partial x_j}$ where $Q^{-1}=(q_{ij}')$, hence it is explicitly described as
\[
2\Delta_Q=\frac{\partial^2}{\partial x_1^2}+\frac{\partial^2}{\partial x_2^2}+\frac{\partial^2}{\partial x_3^2}+\frac{\partial^2}{\partial x_1\, \partial x_2}+\frac{\partial^2}{\partial x_2\, \partial x_3}+\frac{\partial^2}{\partial x_3\, \partial x_1} .
\]
A vector space $\cH_l$ of harmonic polynomials is defined by
\[
\cH_l \coloneqq \{ f\in \cA_l \mid \Delta_Q f=0  \}. 
\]

An action of $\SO(Q)$ on $\cA_l$ is given by $(g\cdot f)(x)=f(xg)$ $(g\in \SO(Q)$, $f\in \cA_l)$, where $x=(x_1,x_2,x_3)$ is regarded as a row vector of degree $3$ and $xg$ means the usual product of matrices. 
Since the operator $\Delta_Q$ and the action of $\SO(Q)$ on $\cA_l$ are commutative, we obtain a $\Q$-rational representation $\rho_l\colon\SO(Q)\to \GL(\cH_l)$. 
It is known that $\rho_l$ is irreducible, and we have $\dim_\Q \cH_l=2l+1$. 
A Lie group $G$ is defined by
\[
G\coloneqq \{g \in \SL_3(\R) \mid gQ\t g=Q\},
\]
which is a compact special orthogonal group of degree $3$. 
We extend $\rho_l$ to a representation $G\to\GL(\cH_l\otimes\C)$. 
Let $\Theta_l(g)\coloneqq \Tr\rho_l(g)$ $(g\in G)$. 
The function $\Theta_l$ on $G$ is called the character function of $\rho_l$. 
It is known that we have
\begin{equation}\label{eq:char1}
\Theta_l(\gamma_\theta)=\frac{e^{i (2l+1) \theta/2}-e^{-i (2l+1)\theta/2}}{e^{i\theta/2}-e^{-i\theta/2}}=\sum_{m=0}^{2l} e^{i\theta (l-m)}
\end{equation}
if $\gamma_\theta$ is a semisimple element whose eigenvalues are $e^{i\theta}$, $1$, and $e^{-i\theta}$. 
Note that $\Theta_l$ is invariant, that is, 
\begin{equation}\label{eq:inv}
\Theta_l(g^{-1}h g)=\Theta_l(h)
\end{equation}
for arbitrary $g$, $h\in G$. 

Let $S_Q\coloneqq \{ x\in \R^{(3)} \mid xQ\t x=1 \}$, where $\R^{(3)}\coloneqq\rv\otimes_\Q \R$. 
Take a $G$-invariant measure $\d x$ on $S_Q$, and normalize it as $\int_{S_Q} \d x=1$. 
Define an inner product $\langle \; , \; \rangle$ on $\cA_l\otimes\C$ by
\begin{equation}\label{eq:inner}
\langle f_1,f_2 \rangle = \int_{S_Q} f_1(x) \overline{f_2(x)} \, \d x \quad (f_1,f_2\in \cA_l).     
\end{equation}
The representation $\rho_l$ of $G$ on $\cH_l\otimes\C$ is unitary for $\langle \; , \; \rangle$. 
Note that each function $f$ in $\cH_l\otimes\C$ is identified with the restriction of $f$ to $S_Q$.

\subsection{Kernel function}

The Legendre polynomial $L_l(t)$ is defined by the following recursion relation:
\[
L_0(t)=1,\quad  L_1(t)=t , \quad L_l(t)=\frac{2l-1}{l} \, t\, L_{l-1}(t)-\frac{l-1}{l}L_{l-2}(t).
\]
For example, $L_2(t)=\frac{3}{2}t^2-\frac{1}{2}$ and $L_3(t) = \frac{5}{2}t^3-\frac{3}{2}t$.
A polynomial $K_l(x,y)$ of six variables $(x,y)$ is defined by
\begin{equation}\label{eq:KL}
K_l(x,y)=\Nm(x)^{l/2}\, \Nm(y)^{l/2}\, L_l(x' Q \, \t y'), \quad x'=\Nm(x)^{-1/2}x, \;\; y'=\Nm(y)^{-1/2}y.    
\end{equation}
Then, it is known that $K_l(x,y)$ is the reproducing kernel on $\cH_l\otimes\C$, that is,
\[
\int_{S_Q}K_l(x,y) \, f(y) \, \d y=f(x) \qquad (f\in \cH_l\otimes \C), 
\]
cf. \cite{Far08}.
Since $L_l(t)$ is a $\Q$-polynomial, the polynomial $K_l(x,y)$ of $x$ belongs to $\cH_l$ for a fixed point $y\in\rv$.

Take a point $a\in S_Q$, and set $G_a\coloneqq \{ g\in G \mid a\, g=a \}$. 
Take a Haar measure $\d g$ (resp. $\d g_a$) on $G$ (resp. $G_a$), and normalize them so that the quotient measure $\d g/\d g_a$ agrees with $\d x$ on $S_Q$ by $x=a\, g^{-1}$. 
Furthermore, the Hilbert space $\cH_l\otimes\C$ is identified with a subspace of $L^2(G/G_a)$ by the mapping
\[
\cH_l\otimes\C \ni f \longrightarrow [g\mapsto f(a \, g^{-1})] \in L^2(G/G_a).
\]
Hence, the function $[(g,h)\mapsto K(a\, g^{-1},a\, h^{-1})]$ belongs to $L^2(G/G_a\times G/G_a)$.
Using $\dim\mathrm{Hom}_{G_a}(\rho_l,\mathbbm{1})=1$, the Peter-Weyl theorem on $L^2(G)$, and the uniqueness of the kernel function, we obatin
\begin{equation}\label{eq:char2}
K_l(a\, g^{-1} ,a\, h^{-1})= \Theta_l(gh^{-1}) \quad (g,h\in G).
\end{equation}

\subsection{Algebraic modular forms}

A lattice $L(\cO)$ in $V_D$ is defined by
\[
L(\cO) \coloneqq \Z \bb_1 +\Z \bb_2 + \Z \bb_3. 
\]
Let $\Gamma\coloneqq \cF(\cO^\times)$. 
Then, $L(\cO)$ is stable for the action of $\Gamma$, since $L(\cO)=\{ x\in\Z+2\cO \mid \Tr(x)=0\}$.
Hence, $\Gamma$ is a finite subgroup of $\SO(Q)\cap \SL_3(\Z)$. 
Furthermore, the generators 
\[
\mathfrak{i}=\cF(i)=\begin{pmatrix} -1&-1&-1 \\ 0&0&1 \\ 0&1&0 \end{pmatrix}, \quad
\mathfrak{j}=\cF(j)=\begin{pmatrix} 0&0&1 \\ -1&-1&-1 \\ 1&0&0 \end{pmatrix}, \quad
\mathfrak{w}=\cF(w)=\begin{pmatrix} 0&0&1 \\ 1&0&0 \\ 0&1&0 \end{pmatrix}.
\]
of $\Gamma$ satisfy the relations
\[
\mathfrak{i}^2=\mathfrak{j}^2=\mathfrak{w}^3=1,\quad \mathfrak{i}\, \mathfrak{j}=\mathfrak{j} \, \mathfrak{i} ,\quad \mathfrak{i}\, \mathfrak{w}=\mathfrak{w} \, \mathfrak{j},\quad \mathfrak{w}\, \mathfrak{i}=\mathfrak{i}\, \mathfrak{j} \, \mathfrak{w}.
\]
Therefore, $\Gamma$ is isomorphic to the alternating group of degree $4$. 
A subspace $\cH_l^\Gamma$ of $\cH_l$ is defined by
\[
\cH_l^\Gamma\coloneqq \{ f\in \cH_l \mid f(x\cdot \gamma)=f(x) \;\; (\forall \gamma\in \Gamma) \}.
\]
The space $\cH_l^\Gamma\otimes \C$ is equipped with the inner product $\langle \; ,\; \rangle$, which was defined in \eqref{eq:inner}. 
A function in $\cH_l^\Gamma$ or $\cH_l^\Gamma\otimes \C$ is called an algebraic modular form of degree $l$ for $\cO$. 

\subsection{Hecke operator and Atkin-Lehner involution}

For a prime $p$, set
\[
\fT_p=\{\cF(x)\in \SO(Q) \mid x\in \cO,\;\; \Nm(x)= p \} .
\]
Since $\fT_p$ is bi-$\Gamma$-invariant and a finite subset of $\SO(Q)$, a linear operator $T_p$ on $\cH_l^\Gamma$ is defined by
\begin{equation*}
    (T_p f)(x)=\sum_{\gamma\in \Gamma\bsl \fT_p} f(x \gamma^{-1}) \quad (f\in \cH_l^\Gamma, \;\; x\in\rv).
\end{equation*}
The opeator $T_p$ commutes with $\Delta_Q$ and $T_q$ (any primes $q\neq p$).
If we fix a $\Q$-basis in $\cH_l^\Gamma$, then $T_p$ is regarded as a square $\Q$-matrix of degree $d$ $(d=\dim \cH_l^\Gamma)$, which is called the Brandt matrix of $T_p$.
The operator $T_p$ is extended to a linear operator on $\cH_l^\Gamma\otimes \C$. 
It is known that the eigenvalues of $T_p$ are in $\R$ and we can simultaneously diagonalize the $T_p$'s for all primes $p$. 
If $f$ is a simultaneous eigenfunction, then $f$ is said to be a Hecke eigenform.

For the case $p=2$, the Hecke operator $T_2$ is a $\Q$-involution on $\cH_l^\Gamma$.
It is identified with the Atkin-Lehner involution for the space of holomorphic newforms.
Set
\[
\cH_{l,+}^\Gamma\coloneqq \{ f\in \cH_l^\Gamma \mid T_2f=(-1)^l f \}, \quad \cH_{l,-}^\Gamma\coloneqq \{ f\in \cH_l^\Gamma \mid T_2f=-(-1)^l f \}.
\]
We have the direct sum
\begin{equation}\label{eq:pmfo}
\cH_{l}^\Gamma=\cH_{l,+}^\Gamma\oplus \cH_{l,-}^\Gamma. 
\end{equation}

\subsection{$L$-function and Quadratic twist}

Take a Hecke eigenform $f\in \cH_{l,\epsilon }^\Gamma\otimes\C$, where $\epsilon=+$ or $-$. 
Denote by $\lambda_{f,p}$ the eigenvalue of $f$ for $T_p$. 
Define the local $L$-factors of $f$ by  
\begin{align*}
L_\infty(s,f) \coloneqq  & 2\, (2\pi)^{-s-l-\tfrac{1}{2}}\, \Gamma\left(s+l+\tfrac{1}{2}\right) \qquad \text{($\Gamma(s)$ is the gamma function)}, \\
L_2(s,f)\coloneqq & (1- \lambda_{f,2} \, 2^{-s})^{-1} \qquad (\lambda_{f,2}=\epsilon (-1)^l),  \\
L_p(s,f)\coloneqq & (1-\lambda_{f,p}\, p^{-s-\frac{1}{2}}+p^{-2s})^{-1} \qquad (p>2).
\end{align*}
It is known that the global $L$-function $L(s,f)\coloneqq\prod_v L_v(s,f)$ is convergent for $\Re(s)>3/2$, $L(s,f)$ is analytically continued to the whole $s$-plane, and we have the functional equation $L(s,f)=\eps(s,f) \, L(1-s,f)$, where $\eps(s,f)$ is an analytic function called the $\eps$-factor of $L(s,f)$. 
These facts can be shown by the Godement-Jacquet theory and others.
In particular, we obtain $\eps(\tfrac{1}{2},f)=\epsilon\, 1$, which is called the root number of $L(s,f)$.

Let $E$ be an imaginary quadratic field with prime discriminant $\Delta_E$, (that is, $\Delta_E\le -3$, $\Delta_E$ is a prime, and $\Delta_E\equiv 1$ or $5\mod 8$).  
Then, we set 
\[
\eta_E(q)\coloneqq  \left( \frac{\Delta_E}{q} \right) \qquad \text{(Legendre symbol)}
\]
for each prime $q$. 
Then, we define the quadratic twist $L(s,f,\eta_E)$ of $L(s,f)$ by $L(s,f,\eta_E)\coloneqq\prod_v L_v(s,f,\eta_E)$ where 
\[
L_\infty(s,f,\eta_E) \coloneqq  L_\infty(s,f), \quad  
L_2(s,f,\eta_E)\coloneqq  (1- \eta_E(2) \, \lambda_{f,2}\, 2^{-s})^{-1} ,  
\]
\[ 
L_p(s,f,\eta_E)\coloneqq  (1-\eta_E(p)\, \lambda_{f,p}\, p^{-s-\frac{1}{2}}+p^{-2s})^{-1} \quad (p>2).
\]
Then, $L(s,f,\eta_E)$ satisfies the same convergence and analytic properties as those of $L(s,f)$, and we have the functional equation $L(s,f,\eta_E)=\eps(s,f,\eta_E) \, L(1-s,f,\eta_E)$. 
When $\eps(\tfrac{1}{2},f)=1$ (i.e., $f\in\cH_{l,+}^\Gamma$), we see that
\begin{equation}\label{eq:epscong}
    \eps(\tfrac{1}{2},f,\eta_E)=1 \quad \text{if and only if} \quad \Delta_E\equiv 5 \mod 8
\end{equation}
see \cite{CW23}*{\S4}. 


\subsection{Galois orbit and Jacquet-Langlands correspondence}\label{sec:GOJL}

Take an algebraic closure $\Bar{\Q}$ of $\Q$ in $\C$. 
The Galois group $\Gal(\Bar{\Q}/\Q)$ acts on $\cA_l\otimes \Bar{\Q}$ as
\[ 
\sigma(f)(x)=\sum_{j_1+j_2+j_3=l} \sigma(a_{j_1,j_2,j_3})\, x_1^{j_1}\, x_2^{j_2}\, x_3^{j_3}
\]
where $f(x)=\sum_{j_1+j_2+j_3=l} a_{j_1,j_2,j_3}\, x_1^{j_1}\, x_2^{j_2}\, x_3^{j_3}\in \cA_l\otimes \Bar{\Q}$. 
The action of $\Gal(\Bar{\Q}/\Q)$ commutes with $\Delta_Q$ and $T_p$ for all primes $p$.

For a Hecke eigenform $f$, the number field $\mathbb{F}_f$ generated by its all Hecke eigenvalues $\lambda_{f,p}$ is called the Hecke field. 
Let $\mathbb{F}_{l,\epsilon}$ be the composite of Hecke fields for all Hecke eigenforms in $\cH_{l,\epsilon}^\Gamma\otimes\C$, where $\epsilon=+$ or $-$. 
Since the eigenvalues $\lambda_{f,p}$ of $T_p$ are in $\R$ and $\cH_{l,\epsilon}^\Gamma$ is stable for any Galois action, we see that $\mathbb{F}_{l,\epsilon}$ is a totally real finite Galois extension of $\Q$. 
Since any Brandt matrix of $T_p$ is diagonalizable over $\mathbb{F}_{l,\epsilon}$, the vector space $\cH_{l,\epsilon}^\Gamma\otimes \mathbb{F}_{l,\epsilon}$ has a basis of Hecke eigenforms. 
For a Hecke eigenform $f\in \cH_{l,\epsilon}^\Gamma$, the subspace spanned by $\sigma f$ for all $\sigma\in\Gal(\mathbb{F}_{l,\epsilon}/\Q)$ is call the Galois orbit of $f$.

Let $S_{2l+2}^{\new,\epsilon}(\Gamma_0(2))$ denote the space of holomorphic newforms of weight $2l+2$ and level $2$ whose Atkin-Lehner eigenvalue is $\epsilon \, 1$. 
It is known that the Jacquet-Langlands correspondence
\[
\mathrm{JL} \colon \cH_{l,\epsilon}^\Gamma\otimes\C \to S_{2l+2}^{\new,\epsilon}(\Gamma_0(2))
\]
is an isomorhism between modules of Hecke algebras, see e.g. \cites{Shi72}. 
Hence, for a Hecke eigenform $f\in \cH_{l,\epsilon}^\Gamma$, the Hecke $L$-function of $\mathrm{JL}(f)$ agrees wtih $L(s,f)$, and the Galois orbit of $\mathrm{JL}(f)$ also agrees with that of $f$.

\section{Numerical computations for dimensions}\label{sec:2}

In this section, we give concrete data for dimensions by the trace formula method. 
The results of this section will be used in \S\ref{sec:fac}, but they are unnecessary for our algorithm in \S\ref{sec:3}. 
All results can be deduced from known results of holomorphic newforms via the Jacquet-Langlands correspondence $\mathrm{JL}$, but our argument here is better since we can see how the information of the discrete sets $\Gamma$ and $\fT_2$ is reflected in the dimension formulas.

\subsection{Trace formula}
Set 
\[
K_l^\Gamma(x,y)\coloneqq \frac{1}{\#\Gamma} \sum_{\gamma\in\Gamma} K_l(x\gamma ,y) .
\]
Then, $K_l^\Gamma(x,y)$ is the reproducing kernel on $\cH_l^\Gamma$. 
Hence, for an orthonormal basis $\psi_1$, $\psi_2,\dots,\psi_d$ in $\cH_l^\Gamma\otimes\C$, we have
\begin{equation}\label{eq:K}
K_l^\Gamma(x,y)=\sum_{j=1}^d \psi_j(x)\, \overline{\psi_j(y)} ,\quad (d=\dim \cH_l^\Gamma).
\end{equation}
Hence, we obtain
\[
    \Tr(T_p|_{\cH_l^\Gamma})= \sum_{j=1}^d \int_{S_Q}(T_p\psi_j)(x)\, \overline{\psi_j(x)}\, \d x =\frac{1}{\#\Gamma}\int_{S_Q} \sum_{\gamma\in \fT_p} K_l(x\gamma^{-1} ,x) \, \d x .
\]
The set of $\Gamma$-conjugacy classes in $\fT_p$ is written as
\[
\left\{ [\gamma_j] \mid 1\le j\le m_p \right\} \quad \left([\gamma_j]=\{\delta^{-1}\gamma_j\delta \mid \delta\in\Gamma \} \right).
\]
Take a real number $\theta_j$ such that $\theta_j$, $1$, and $-\theta_j$ are the eigenvalues of $\gamma_j$. 
By exchanging the integral and the sum, and using \eqref{eq:inv} and \eqref{eq:char2}, we derive
\begin{equation}\label{eq:trace}
    \Tr(T_p|_{\cH_l^\Gamma})=\frac{1}{\#\Gamma}\sum_{j=1}^{m_p} \# [\gamma_j] \, \Theta_l(\gamma_{\theta_j}). 
\end{equation}

\subsection{Dimension formula}

The $\Gamma$-conjugacy classes $[\gamma]$ in $\fT_1=\Gamma$ are as follows:
\[
[1], \quad [\mathfrak{i}], \quad [\mathfrak{w}] , \quad [\mathfrak{w}^2]. 
\]
Their cardinalities are $1$, $3$, $4$, $4$ respectively.  
For the trivial operator $T_1$, we see from \eqref{eq:trace}
\[
\dim \cH_l^\Gamma=\Tr(T_1|_{\cH_l^\Gamma})= \frac{1}{\#\Gamma}\sum_{[\gamma]\in \{ [1], [\mathfrak{i}],  [\mathfrak{w}] ,  [\mathfrak{w}^2] \} } \#[\gamma] \, \Theta_l(\gamma).
\]
Hence, we get
\[
\dim \cH_l^\Gamma=\frac{\Theta_l(1)}{12}+\frac{\Theta_l(\gamma_{\pi})}{4}+\frac{2\, \Theta_l(\gamma_{2\pi/3})}{3}.
\]
By \eqref{eq:char1} we obtain
\[
\dim \cH_l^\Gamma = \frac{2l+1}{12}+\frac{1}{4}(-1)^l +\frac{2}{3} \times\begin{cases}
1 & \text{if $l\equiv 0 \mod 3$,} \\
0 & \text{if $l\equiv 1 \mod 3$,} \\
-1 & \text{if $l\equiv 2 \mod 3$.} 
\end{cases}
\]
A list for numerical numbers is as follows:
\[
\begin{array}{c|ccccccccccccc}
l & 0& 1& 2 & 3 & 4 & 5 & 6 & 7 & 8 & 9 & 10& 11& 12  \\ \hline
\dim \cH_l^\Gamma  & 1& 0&  0 & 1& 1& 0& 2& 1& 1& 2& 2& 1& 3
\end{array}
\]
In addition,
\[
\sum_{l=0}^\inf\dim \cH_l^\Gamma \, t^l =\frac{1+t^6}{(1-t^3)(1-t^4)}. 
\]
If we divide it by even and odd degrees, then
\begin{equation}\label{eq:dim1}
\sum_{m=0}^\inf\dim \cH_{2m}^\Gamma \, t^{2m} =\frac{1+t^6}{(1-t^4)(1-t^6)}, \qquad \sum_{m=0}^\inf\dim \cH_{2m+1}^\Gamma \, t^{2m+1} =\frac{t^3(1+t^6)}{(1-t^4)(1-t^6)}. 
\end{equation}

\subsection{Trace of Atkin-Lehner operator}

By a direct calculation, we have
\[
\fT_2=\Gamma \gamma_1=\gamma_1\Gamma=\Gamma\gamma_2=\gamma_2\Gamma,
\]
\[
\gamma_1=\cF(1+i)=\begin{pmatrix}
    0&-1&0 \\ 1&1&1 \\ -1&0&0 
\end{pmatrix} ,\quad \gamma_2=\cF(i-j)=\begin{pmatrix}
    0&-1&0 \\ -1&0&0 \\ 0&0&-1 
\end{pmatrix},
\]
and there are only two $\Gamma$-conjugacy classes $[\gamma_1]$ and $[\gamma_2]$ in $\fT_2$ $(\#[\gamma_1]=\#[\gamma_2]=6)$.
Hence, from \eqref{eq:trace} we deduce
\[
\Tr(T_2|_{\cH_l^\Gamma})=\frac{1}{2}\Theta_l(\gamma_{\pi /2})+\frac{1}{2}\Theta_l(\gamma_{\pi})=\begin{cases}
    0 & \text{$l\equiv 1$ or $2 \mod 4$}, \\ -1 & \text{$l\equiv 3 \mod 4$}, \\ 1 & \text{$l\equiv 0 \mod 4$}.
\end{cases}
\]
Since
\[
\Tr(T_2|_{\cH_l^\Gamma})=(-1)^l\dim \cH_{l,+}^\Gamma-(-1)^l\dim \cH_{l,-}^\Gamma, \quad \dim \cH_l^\Gamma=\dim \cH_{l,+}^\Gamma+\dim \cH_{l,-}^\Gamma,
\]
we have
\begin{multline*}
\dim \cH_{l,\pm}^\Gamma =\frac{\dim \cH_l^\Gamma\pm(-1)^l\Tr(T_2|_{\cH_l^\Gamma})}{2}\\
=\frac{2l+1}{24}+\frac{1}{8}(-1)^l +\frac{1}{3} \times\begin{cases}
1 & \text{if $l\equiv 0 \mod 3$,} \\
0 & \text{if $l\equiv 1 \mod 3$,} \\
-1 & \text{if $l\equiv 2 \mod 3$,} 
\end{cases}
\pm\frac{1}{2} \times\begin{cases}
1 & \text{if $l\equiv 0$ or $3 \mod 4$,} \\
0 & \text{if $l\equiv 1$ or $2 \mod 4$.}  
\end{cases}
\end{multline*}
A list for numerical numbers is as follows:
\[
\begin{array}{c|ccccccccccccc}
l & 0& 1& 2 & 3 & 4 & 5 & 6 & 7 & 8 & 9 & 10& 11& 12  \\ \hline
\dim \cH_l^\Gamma & 1& 0&  0 & 1& 1& 0& 2& 1& 1& 2& 2& 1& 3 \\
\dim \cH_{l,+}^\Gamma & 1& 0& 0 & 1& 1& 0& 1& 1& 1& 1& 1& 1& 2 \\
\dim \cH_{l,-}^\Gamma & 0& 0& 0 & 0& 0& 0& 1& 0& 0& 1& 1& 0& 1
\end{array}
\]
Since the dimension changes modulo 12, we obtain
\[
\dim \cH_{12m+j,\pm}^\Gamma = m+ \dim \cH_{j,\pm}^\Gamma \quad \text{if $l=12m+j$, $m\ge 0$ and $0\le j< 12$}. 
\]
Their generating functions are as follows:
\begin{equation}\label{eq:dim2}
\sum_{l=0}^\inf\dim \cH_{l,+}^\Gamma \, t^l =\frac{1}{(1-t^3)(1-t^4)},\qquad \sum_{l=0}^\inf\dim \cH_{l,-}^\Gamma \, t^l =\frac{t^6}{(1-t^3)(1-t^4)}.    
\end{equation}

\section{Computations for basis of \texorpdfstring{$\cH_{l,\pm}^\Gamma$}{TEXT}}\label{sec:3}

In this section, we present our algorithm for bases of $\cH_l^\Gamma$ and $\cH_{l,\pm}^\Gamma$. 
Throughout this paper, we write $\Q^n$ for the space of column vectors of degree $n$ over $\Q$. 
Write $e_j=e_j^{(n)}$ for the $j$-th basic vector in $\Q^n$, that is, $e_j^{(n)}={}^t\!\begin{pmatrix}
    \delta_{j1} & \delta_{j2} &\cdots& \delta_{jn}
\end{pmatrix}$ where $\delta_{ij}$ means the Dirac delta. 
Write $e_0=e_0^{(n)}$ for the zero vector, that is, $e_0={}^t\!\begin{pmatrix}
    0 & 0 &\cdots& 0
\end{pmatrix}$. 
Throughout this section, we also consider $\Delta_Q$ to be a linear mapping from $\cA_l$ to $\cA_{l-2}$.

\subsection{Monomials of $\cA_l$}

The monomials $x_1^{l_1}x_2^{l_2}x_3^{l_3}$ form a basis of $\cA_l$, hence we will order the monomials. 
Our order is known as the graded reverse lexicographic order. 
Put $d_l\coloneqq\dim\cA_l$ and set
\[
\cL_l\coloneqq\{(l_1,l_2,l_3)\in\Z^3_{\ge0} \mid l_1+l_2+l_3=l \}, \qquad \cN_l\coloneqq\{ m\in\Z \mid 0\le m\le d_l-1 \}.
\]
Note that $\dim \cA_l=\frac{(l+1)(l+2)}{2}$ and $\cN_l$ has the usual order. 
An order $<$ on $\cL_l$ is defined by 
\[
\ul<\ul'\quad \Leftrightarrow \quad \text{$l_1<l_1'$ or [$l_1=l_1'$ and $l_2<l_2'$]}
\]
for $\ul=(l_1,l_2,l_3)$, $\ul'=(l_1',l_2',l_3')\in\cL_l$, and a mapping $\fF_l\colon \cL_l\to \cN_l$ is defined by
\[
\fF_l(\ul)=l_1(l+1)-\frac{l_1(l_1-1)}{2}+l_2. 
\]
\begin{lem}
The mapping $\fF_l\colon \cL_l\to \cN_l$ is an ordered isomorphism. Here, we are considering the usual ordering on $\cN_l$.  
\end{lem}
\begin{proof}
First, we will show that $\fF_l$ is monotonic increase. 
It is obvious that $\fF_l$ is monotonic increase when $l_1$ (resp. $l_2$) moves and $l_2$ (resp. $l_1$) is fixed. 
Hence, it is enough to confirm
\[
\fF_l(l_1+1,0,l-l_1-1)-\fF_l(l_1,l-l_1,0)=(l_1+1)(l+1)-\frac{l_1(l_1+1)}{2}-l_1(l+1)+\frac{l_1(l_1-1)}{2}-l+l_1=1.
\]
Therefore, we find that $\fF_l$ is monotonic increase, hence $\fF_l$ is injective. 
Since the cardinality of $\cL_l$ is the same as $\cN_l$, we obtain the surjectivity of $\fF_l$.  
\end{proof}
For each $m\in\cN_l$, we set
\[
l_1(m)\coloneqq \max \left\{ l_1\in\N  \; \middle| \;  l_1(l+1)-\frac{l_1(l_1-1)}{2} \le m \right\},
\]
\[
l_2(m)\coloneqq m-l_1(m)(l+1)+\frac{l_1(m)(l_1(m)-1)}{2}, \quad l_3(m)\coloneqq l-l_1(m)-l_2(m).
\]
Then, the inverse mapping $\fF_l^{-1}\colon\cN_l\to\cL_l$ is given by $\fF_l^{-1}(m)=(l_1(m),l_2(m),l_3(m))$.

Set $\underline{x}^{\ul}\coloneqq x_1^{l_1}x_2^{l_2}x_3^{l_3}$. 
Now, we can consider the basis $\{ \underline{x}^{\ul} \}_{\ul\in\cL_l}$ whose vectors are arranged in the descending order of $\cL_l$, that is.
\[
\left\{ \, x_3^l, \;\; x_2 \, x_3^{l-1}, \;\; x_2^2x_3^{l-2},  \; \dots, \;  x_2^{l-1}x_3, \;\; x_2^l, \;\;  x_1x_3^{l-1}, \;\; x_1x_2x_3^{l-2},\dots \, \right\} . 
\]
Fix a basis of $\Q^{d_l}$ as
\[
\left\{\, e_1^{(d_l)}, \;\; e_2^{(d_l)}, \; \dots \; , \;\;  e_{d_l}^{(d_l)} \, \right\}.
\]
Then, from the ordered isomorphism $\fF_l$, we obtain an isomorphism $\sF_l$ from $\cA_l$ to $\Q^{d_l}$ as
\[
\sF_l\colon \cA_l\to \Q^{d_l} ,\qquad \sF_l(\underline{x}^{\ul})=e_{\fF_l(\ul)+1}. 
\]

\subsection{Harmonic polynomials}
Define a set $\tilde\cL_l$ (including $\cL_l$) as
\[
\tilde\cL_l \coloneqq \{(l_1,l_2,l_3)\in\Z^3 \mid l_1+l_2+l_3=l \}.
\]
Maps $\fD_i\colon \tilde\cL_l\to \tilde\cL_{l-1}$, which mean differentiations, are defined by 
\[
\fD_1(l_1,l_2,l_3)=(l_1-1,l_2,l_3), \;\; \fD_2(l_1,l_2,l_3)=(l_1,l_2-1,l_3), \;\; \fD_3(l_1,l_2,l_3)=(l_1,l_2,l_3-1) ,
\]
and we also set
\[
\fD_{ij}\colon \cL_l\to\tilde\cL_{l-2} ,\quad \fD_{ij}\coloneqq \fD_i \circ \fD_j.
\]
Set $\tilde\cN_l\coloneqq \cN_l\sqcup\{-1\}$ and define
\[
\tilde\fF_l\colon \tilde\cL_l\to \tilde\cN_l ,\qquad \tilde\fF_l|_{\cL_l}=\fF_l, \qquad \tilde\fF_l(\ul)=-1 \quad (\ul\in\tilde\cL_l\setminus\cL_l).
\]
From this, let us define the second derivative. 
Consider the composite mapping $\fd_{ij}$ as
\[
\fd_{ij}\colon \cN_l \to \cL_l\to \tilde\cL_{l-2} \to\tilde\cN_{l-2},\qquad \fd_{ij}= \tilde\fF_{l-2}\circ \fD_{ij} \circ\fF_l^{-1}.
\]
A linear mapping $\partial_{ij}\colon \Q^{d_l} \to \Q^{d_{l-2}}$ is defined by 
\[
\partial_{jj}(e_{m+1})=l_j(m)\, (l_j(m)-1)\, e_{\fd_{jj}(m)+1}, \quad \partial_{ij}(e_{m+1})=l_i(m)\, l_j(m)\, e_{\fd_{ij}(m)+1},
\]
where $m \in \cN_l$. 
Then, a differential operator $\Delta_l\colon  \Q^{d_l} \to \Q^{d_{l-2}}$ is defined by
\[
\Delta_l = \frac{1}{2} \left( \partial_{11}+\partial_{22}+\partial_{33}+\partial_{12}+\partial_{13}+\partial_{23} \right) .
\]
Through the isomorphism $\sF_l$ between $\Q^{d_l}$ and $\cA_l$, we identify $\Delta_l$ with $\Delta_Q$. 
Therefore, the kernel $\mathrm{Ker}(\Delta_l)$ agrees with $\sF_l(\cH_l)$.  
Note that the size of the representative matrix of $\Delta_l$ is $d_{l-2}\times d_l$.

\subsection{Average mapping \texorpdfstring{$\mathscr{A}_l$}{TEXT}}

An action of $\GL_3(\Q)$ on $\cA_l$ is defined by
\begin{equation*}
(\gamma \cdot f)(x_1,x_2,x_3)=f((x_1,x_2,x_3)\gamma) \quad (\gamma\in\GL_3(\Q)).    
\end{equation*}
Note that $(x_1,x_2,x_3)$ is a row vector of degree $3$ and $(x_1,x_2,x_3)\gamma$ means the usual product of matrices. 
We also note that this action of $\GL_3(\Q)$ on $\cA_l$ is a left action, that is, $(\gamma_1\gamma_2)\cdot f=\gamma_1\cdot(\gamma_2\cdot f)$. 
An average (linear) mapping $\mathrm{Ave}_l \colon \cA_l \to \cA_l$ of $\Gamma=\cF(\cO^\times)$ is defined by
\[
\mathrm{Ave}_l (f)=\sum_{\gamma\in\Gamma} \gamma\cdot f \qquad (f\in\cA_l). 
\]
For any $\Gamma$-invariant polynomial $f$ in $\cA_l$, we have $\mathrm{Ave}_l (f)= \#\Gamma \times f$. 
Hence, we find that the image of $\mathrm{Ave}_l$ is the subspace consisting of $\Gamma$-invariant polynomials in $\cA_l$. 
Therefore, we obtain
\begin{equation}\label{eq:h=avde}
\cH_l^\Gamma= \mathrm{Im}(\mathrm{Ave}_l) \, \cap \, \mathrm{Ker}(\Delta_Q),    
\end{equation}
since $\cH_l=\mathrm{Ker}(\Delta_Q)$. 
In the computational program for $\mathrm{Ave}_l$, we directly input the actions of $\Gamma$ on the polynomials. 
See \S\ref{sec:list} for the list of elements of $\Gamma$.

\subsection{Computation of a basis of \texorpdfstring{$\cH_l^\Gamma$}{TEXT}}\label{sec:algH}

Define a linear mapping $\mathscr{A}_l\colon \Q^{d_l}\to\Q^{d_l}$ by
\[
\mathscr{A}_l=\sF_l\circ \mathrm{Ave}_l\circ \sF_l^{-1}.
\]
Then, by \eqref{eq:h=avde} we have
\begin{equation}\label{eq:intersection}
\sF_l(\cH_l^\Gamma)= \mathrm{Im}(\mathscr{A}_l) \, \cap \, \mathrm{Ker}(\Delta_l).
\end{equation}
From this, we consider an algorithm on a basis of $\sF_l(\cH_l^\Gamma)$ by the right-hand side of \eqref{eq:intersection}. 
\begin{itemize}
\item Denote by $A$ the representative matrix of $\mathscr{A}_l \colon \Q^{d_l}\to\Q^{d_l}$ with respect to the basis $ $ $\{e_1,e_2,\dots, e_{d_l} \}$. The size of $A$ is $d_l\times d_l$. 
\item Denote by $D$ the representative matrix of $\Delta_l\colon \Q^{d_l}\to\Q^{d_{l-2}}$ with respect to the bases $\{e_1,e_2,\dots, e_{d_l} \}$ and $\{e_1,e_2,\dots, e_{d_{l-2}} \}$. The size of $D$ is $d_{l-2}\times d_l$.  
\end{itemize}
Suppose that $A$ and $D$ were numerically obtained. 
Set
\[
A=\begin{pmatrix}\ba_1&\ba_2&\cdots&\ba_{d_l} \end{pmatrix}.
\]
Since $\mathrm{Im}(\mathscr{A}_l)=\langle \ba_1, \ba_2,\dots,\ba_{d_l} \rangle$, by simplifying the matrix $A$, 
one can calculate a basis of $\mathrm{Im}(\mathscr{A}_l)$ as
\[
\ba_{i_1}, \ba_{i_2},\dots,\ba_{i_c} \qquad (i_1<i_2<\cdots<i_c , \quad  c=\mathrm{rank}(A)) .
\]
Set
\[
B\coloneqq  \begin{pmatrix} \ba_{i_1}& \ba_{i_2}&\cdots&\ba_{i_c} \end{pmatrix}.
\]
Then, $B$ is a $d_l\times c$ matrix, and 
\[
T_B \colon \Q^c \to \mathrm{Im}(\mathscr{A}_l) \, (\subset \Q^{d_l}),\quad T_B(x)=Bx
\]
is isomorphic. 
Consider the composition
\[
\Delta_l\circ T_B \colon \Q^c \to \Q^{d_{l-2}} ,\quad \Delta_l\circ T_B(x)=DBx \quad (c=\mathrm{rank}(B)).
\]
Since the kernel of $\Delta_l\circ T_B$ is the solution space of $DBx=0$, one can compute its basis and we denote it by 
\[
\bb_1,\bb_2,\dots,\bb_t \qquad (t=c-\mathrm{rank}(DB)).
\]
Since $T_B$ is isomorphic, a basis of \eqref{eq:intersection} is given by
\[
B\bb_1,B\bb_2,\dots,B\bb_t.
\]
Thus, a basis of $\cH_l^\Gamma$ is obtained by
\[
 \sF_l^{-1}(B\bb_1),\sF_l^{-1}(B\bb_2),\dots,\sF_l^{-1}(B\bb_t).
\]
This completes the description of our algorithm.

\subsection{Average mapping \texorpdfstring{$\mathrm{Ave}_{l,\pm}$}{TEXT}}

Recall the Hecke operator $T_2$ and the corresponding set $\fT_2=\gamma_2\Gamma=\Gamma\gamma_2$. 
An average (linear) mapping $\mathrm{Ave}_{l,\pm}\colon \cA_l\to \cA_l$ is defined by
\[
\mathrm{Ave}_{l,\pm}(f)\coloneqq \sum_{\gamma\in \Gamma} \gamma\cdot f \pm (-1)^l\sum_{\gamma\in \Gamma \gamma_2} \gamma\cdot f \qquad \text{(double sign in same order)}
\]
where $f\in \cA_l$. 
\begin{lem}\label{lem:pm}
\[
\cH_{l,\pm}^\Gamma= \mathrm{Im}(\mathrm{Ave}_{l,\pm}) \, \cap \, \mathrm{Ker}(\Delta_Q).
\]    
\end{lem}
\begin{proof}
We can extend $\mathrm{Ave}_{l,\pm}$ to a linear mapping $\cA_l\otimes\C\rightarrow\cA_l\otimes\C$, and we have the orthogonal decomposition
\[
\cA_l\otimes\C = (\cH_l\otimes \C) \bigoplus (\cH_l\otimes \C)^\perp
\]
with respect to the inner product \eqref{eq:inner}. 
Since $\cH_l\otimes \C$ (resp. $(\cH_l\otimes \C)^\perp$) is stable for the action of $G$ and $\Gamma\sqcup\Gamma\gamma_2$ is a subset of $\SO(Q)\subset G$, we have $\mathrm{Ave}_{l,\pm}(\cH_l\otimes \C)\subset \cH_l\otimes \C$ (resp. $\mathrm{Ave}_{l,\pm}((\cH_l\otimes \C)^\perp)\subset (\cH_l\otimes \C)^\perp$). 
Hence,
\begin{equation}\label{eq:cdecom}
\mathrm{Ave}_{l,\pm}(\cA_l\otimes\C)\cap (\cH_l\otimes\C)=\mathrm{Ave}_{l,\pm}(\cH_l\otimes \C)=\mathrm{Ave}_{l,\pm}(\cH_l)\otimes \C.    
\end{equation}
Since $\Delta_Q$ commutes with $\mathrm{Ave}_{l,\pm}$, we have $\mathrm{Ave}_{l,\pm}(\cH_l)\subset \mathrm{Ave}_{l,\pm}(\cA_l)\cap \cH_l$. 
Fix a $\Q$-basis $u_1,\dots,u_m$ of $\mathrm{Ave}_{l,\pm}(\cH_l)$. 
Since the $\Q$-polynomials $u_1,\dots,u_m$ are linearly independent over $\C$, any $\Q$-polynomial in $\mathrm{Ave}_{l,\pm}(\cA_l)\cap \cH_l$ is expressed as a $\Q$-linear combination of $u_1,\dots,u_m$ by \eqref{eq:cdecom}. 
Therefore, we see $\mathrm{Im}(\mathrm{Ave}_{l,\pm})\cap \cH_l = \mathrm{Im} (\mathrm{Ave}_{l,\pm}|_{\cH_l})$, that is, 
\[
\mathrm{Im}(\mathrm{Ave}_{l,\pm})\cap \mathrm{Ker}(\Delta_Q) = \mathrm{Im} (\mathrm{Ave}_{l,\pm}|_{\cH_l}),
\]
because $\cH_l=\mathrm{Ker}(\Delta_Q)$. 
Hence, it is sufficient to prove $\mathrm{Im}(\mathrm{Ave}_{l,\pm}|_{\cH_l})=\cH_{l,\pm}^\Gamma$. 

For any polynomial $f\in\cA_l$, we have
\begin{multline}\label{eq:avav}
   \frac{1}{\#\Gamma} \mathrm{Ave}_{l,\pm}\circ \mathrm{Ave}_{l}(f) = \frac{1}{\#\Gamma} \sum_{\gamma\in \Gamma} \gamma\cdot \left(\sum_{\gamma'\in\Gamma}\gamma'\cdot f \right) \pm \frac{(-1)^l}{\#\Gamma}\sum_{\gamma\in \Gamma \gamma_2} \gamma\cdot \left(\sum_{\gamma'\in\Gamma}\gamma'\cdot f \right) \\ 
   = \frac{1}{\#\Gamma} \sum_{\gamma\in \Gamma} \sum_{\gamma'\in\Gamma}(\gamma\gamma')\cdot f  \pm \frac{(-1)^l}{\#\Gamma}\sum_{\gamma\in \gamma_2\Gamma } \sum_{\gamma'\in\Gamma}(\gamma\gamma')\cdot f  = \sum_{\gamma\in \Gamma} \gamma\cdot f  \pm (-1)^l\sum_{\gamma\in \gamma_2\Gamma } \gamma\cdot f = \mathrm{Ave}_{l,\pm}(f). 
\end{multline}
In addition, by \eqref{eq:h=avde}, $\frac{1}{\#\Gamma}\mathrm{Ave}_{l}|_{\cH_l}$ agrees with the projection $\mathrm{Proj}\colon \cH_l\to\cH_l$ such that $\mathrm{Im}(\mathrm{Proj})=\cH_l^\Gamma$ and $\mathrm{Proj}|_{\cH_l^\Gamma}$ is the identity mapping, since any $\Gamma$-invariant polynomial $f$ satisfies $\frac{1}{\#\Gamma}\mathrm{Ave}_{l}(f)=f$. 
Therefore, from \eqref{eq:avav} we obtain
\[
\mathrm{Ave}_{l,\pm}|_{\cH_l}=\mathrm{Ave}_{l,\pm}|_{\cH_l}\circ\mathrm{Proj}=\mathrm{Ave}_{l,\pm}|_{\cH_l^\Gamma}\circ\mathrm{Proj}. 
\]
By \eqref{eq:pmfo}, 
we see that $\mathrm{Ave}_{l,\pm}|_{\cH_l^\Gamma}$ agrees with the projection $\mathrm{Proj}_\pm\coloneqq\frac{1}{2}(\mathrm{Id} +(-1)^l T_2)$ on $\cH_l^\Gamma$ such that $\mathrm{Im}(\mathrm{Proj}_\pm)=\cH_{l,\pm}^\Gamma$ and $\mathrm{Proj}_\pm|_{\cH_{l,\pm}^\Gamma}$ is the identity mapping. 
This completes the proof.     
\end{proof}

\subsection{Computation of a basis of \texorpdfstring{$\cH_{l,\pm}^\Gamma$}{TEXT}}

By Lemma \ref{lem:pm}, we can consider an algorithm for $\cH_{l,\pm}^\Gamma$ in the same way as in \S\ref{sec:algH} for $\cH_{l}^\Gamma$.
Define a linear mapping $\mathscr{A}_{l,\pm}\colon \Q^{d_l}\to\Q^{d_l}$ by
\[
\mathscr{A}_{l,\pm}=\sF_l\circ \mathrm{Ave}_{l,\pm}\circ \sF_l^{-1}.
\]
Then, by Lemma \ref{lem:pm} we have
\begin{equation}\label{eq:intersectionpm}
\sF_l(\cH_{l,\pm}^\Gamma)= \mathrm{Im}(\mathscr{A}_{l,\pm}) \, \cap \, \mathrm{Ker}(\Delta_l).
\end{equation}
Denote by $A_\pm$ the representative matrix of $\mathscr{A}_{l,\pm} \colon \Q^{d_l}\to\Q^{d_l}$ with respect to the basis $ $ $\{e_1,e_2,\dots, e_{d_l} \}$. The size of $A_\pm$ is $d_l\times d_l$. 
We use the same notation $D$ as in \S\ref{sec:algH}, and suppose that $A_\pm$ and $D$ were numerically obtained. 
Set
\[
A_\pm=\begin{pmatrix}\ba_1^\pm&\ba_2^\pm&\cdots&\ba_{d_l}^\pm \end{pmatrix}.
\]
By simplifying the matrix $A_\pm$, we obtain a basis of $\mathrm{Im}(\mathscr{A}_l)$ as
\[
\ba_{i_1}^\pm, \ba_{i_2}^\pm,\dots,\ba_{i_{c_\pm}}^\pm \qquad (i_1<i_2<\cdots<i_{c_\pm} , \quad  c_\pm=\mathrm{rank}(A_\pm)) .
\]
Set $B_\pm\coloneqq  \begin{pmatrix} \ba_{i_1}^\pm& \ba_{i_2}^\pm&\cdots&\ba_{i_{c_\pm}}^\pm \end{pmatrix}$, and we have an isomorphism $T_{B_\pm} \colon \Q^{c_\pm} \to \mathrm{Im}(\mathscr{A}_{l,\pm})$, $T_{B_\pm}(x)=B_\pm x$. 
Consider the composition
\[
\Delta_l\circ T_{B_\pm} \colon \Q^{c_\pm} \to \Q^{d_{l-2}} ,\quad \Delta_l\circ T_{B_\pm}(x)=DB_\pm x .
\]
One can compute a basis of $\mathrm{Ker}(\Delta_l\circ T_{B_\pm})$ and we denote it by $\bb_1^\pm$, $\bb_2^\pm,\dots,\bb_{t_\pm}^\pm$ where $t_\pm=c_\pm-\mathrm{rank}(DB_\pm)$.
Therefore, a basis of $\mathrm{Im}(\mathscr{A}_{l,\pm}) \, \cap \, \mathrm{Ker}(\Delta_l)$ is given by $B_\pm\bb_1^\pm$, $B_\pm\bb_2^\pm,\dots,B_\pm\bb_{t_\pm}^\pm$. 
Thus, a basis of $\cH_{l,\pm}^\Gamma$ is obtained by $\sF_l^{-1}(B_\pm\bb_1^\pm)$, $\sF_l^{-1}(B_\pm\bb_2^\pm),\dots,\sF_l^{-1}(B_\pm\bb_{t_\pm}^\pm)$.
This completes the description of our algorithm.

\subsection{List of elements of $\Gamma$ and $\Gamma\gamma_2$}\label{sec:list}

A list of elements of $\Gamma$ and $\Gamma\gamma_2$ is written here. 
It is necessary to input them for the program to compute the average mapping.
The elements of $\Gamma$ are as follows:
\[
1=\begin{pmatrix} 1&0&0 \\ 0&1&0 \\ 0&0&1 \end{pmatrix} ,\quad \mathfrak{w}=\begin{pmatrix} 0&0&1 \\ 1&0&0 \\ 0&1&0 \end{pmatrix} ,\quad \mathfrak{w}^2=\begin{pmatrix} 0&1&0 \\ 0&0&1 \\ 1&0&0 \end{pmatrix},
\]
\[
\mathfrak{i}=\begin{pmatrix} -1&-1&-1 \\ 0&0&1 \\ 0&1&0 \end{pmatrix} ,\quad \mathfrak{i}\mathfrak{w}=\begin{pmatrix} -1&-1&-1 \\ 0&1&0 \\ 1&0&0 \end{pmatrix},\quad \mathfrak{i}\mathfrak{w}^2=\begin{pmatrix} -1&-1&-1 \\ 1&0&0 \\ 0&0&1 \end{pmatrix},
\]
\[
\mathfrak{j}=\begin{pmatrix} 0&0&1 \\ -1&-1&-1 \\ 1&0&0 \end{pmatrix}, \quad
\mathfrak{j}\mathfrak{w}=\begin{pmatrix} 0&1&0 \\ -1&-1&-1 \\ 0&0&1 \end{pmatrix}, \quad
\mathfrak{j}\mathfrak{w}^2=\begin{pmatrix} 1&0&0 \\ -1&-1&-1 \\ 0&1&0 \end{pmatrix}, 
\]
\[
\mathfrak{i}\mathfrak{j}=\begin{pmatrix} 0&1&0 \\ 1&0&0 \\ -1&-1&-1 \end{pmatrix}, \quad
\mathfrak{i}\mathfrak{j}\mathfrak{w}=\begin{pmatrix} 1&0&0 \\ 0&0&1 \\ -1&-1&-1 \end{pmatrix}, \quad
\mathfrak{i}\mathfrak{j}\mathfrak{w}^2=\begin{pmatrix} 0&0&1 \\ 0&1&0 \\ -1&-1&-1 \end{pmatrix}.
\]
The elements of $\Gamma\gamma_2$ are as follows:
\[
\gamma_2=\begin{pmatrix} 0&-1&0 \\ -1&0&0 \\ 0&0&-1 \end{pmatrix} ,\quad 
\mathfrak{w}\gamma_2=\begin{pmatrix} 0&0&-1 \\ 0&-1&0 \\ -1&0&0 \end{pmatrix} ,\quad 
\mathfrak{w}^2\gamma_2=\begin{pmatrix} -1&0&0 \\ 0&0&-1 \\ 0&-1&0 \end{pmatrix},
\]
\[
\mathfrak{i}\gamma_2=\begin{pmatrix} 1&1&1 \\ 0&0&-1 \\ -1&0&0 \end{pmatrix} ,\quad 
\mathfrak{i}\mathfrak{w}\gamma_2=\begin{pmatrix} 1&1&1 \\ -1&0&0 \\ 0&-1&0 \end{pmatrix},\quad 
\mathfrak{i}\mathfrak{w}^2\gamma_2=\begin{pmatrix} 1&1&1 \\ 0&-1&0 \\ 0&0&-1 \end{pmatrix},
\]
\[
\mathfrak{j}\gamma_2=\begin{pmatrix} 0&0&-1 \\ 1&1&1 \\ 0&-1&0 \end{pmatrix}, \quad
\mathfrak{j}\mathfrak{w}\gamma_2=\begin{pmatrix} -1&0&0 \\ 1&1&1 \\ 0&0&-1 \end{pmatrix}, \quad
\mathfrak{j}\mathfrak{w}^2\gamma_2=\begin{pmatrix} 0&-1&0 \\ 1&1&1 \\ -1&0&0 \end{pmatrix}, 
\]
\[
\mathfrak{i}\mathfrak{j}\gamma_2=\begin{pmatrix} -1&0&0 \\ 0&-1&0 \\ 1&1&1 \end{pmatrix}, \quad
\mathfrak{i}\mathfrak{j}\mathfrak{w}\gamma_2=\begin{pmatrix} 0&-1&0 \\ 0&0&-1 \\ 1&1&1 \end{pmatrix}, \quad
\mathfrak{i}\mathfrak{j}\mathfrak{w}^2\gamma_2=\begin{pmatrix} 0&0&-1 \\ -1&0&0 \\ 1&1&1 \end{pmatrix}.
\]

\section{Application to non-vanishing theorems}\label{sec:4}

In this section, we present an application of our computational results (cf. Appendix \ref{sec:app}).  

From \S\ref{sec:Q} we have
\[
\Nm(x)=\begin{pmatrix}
    x_1&x_2&x_3
\end{pmatrix}Q\begin{pmatrix}
    x_1\\ x_2\\x_3
\end{pmatrix}=3(x_1^2+x_2^2+x_3^2)-2(x_1x_2+x_2x_3+x_3x_1), 
\]
where $x=x_1\bb_1+x_2\bb_2+x_3\bb_3\in V_D$. 
Let $E$ be an imaginary quadratic field over $\Q$, and denote by $\Delta_E$ the discriminant of $E$. 
When $x_E \in L(\cO)$ satisfies $\Nm(x_E)=-\Delta_E$, the point $x_E$ is said to be a CM point of $E$.

Suppose that $\Delta_E\equiv 5 \mod 8$. 
Let $h_E$ denote the class number of $E$ and $\fo_E$ the integer ring of $E$. 
By \cite{CW23}*{\S 3 and \S 4} there exist CM points $a_j$ $(1\le j\le h_E)$ of $E$ such that, for any $f\in\cH_l^\Gamma$, the finite sum
\[
\mathfrak{P}_E(f) \coloneqq \sum_{j=1}^{h_E} f(a_j) 
\]
equals Waldspurger's toric period of $f$ up to non-zero constant. 
Hence, if $f$ is a Hecke eigenform, then $\mathfrak{P}_E(f)\neq 0$ implies $L(\tfrac{1}{2},f,\eta_E)\neq 0$.
Set $a_j=a_{j1}\bb_1+a_{j2}\bb_2+a_{j3}\bb_3$ $(a_{j1}$, $a_{j2}$, $a_{j3}\in\Z)$. 
Since $1\equiv \Delta_E=-\Nm(a_j)\equiv (a_{j1}+a_{j2}+a_{j3})^2 \mod 4$, we obtain 
\begin{equation}\label{eq:j1j2j3}
    a_{j1}+a_{j2}+a_{j3}\equiv 1\mod 2.  
\end{equation}

Now, the space $\cH^\Gamma_{3,+}$ is generated by a single polynomial $f_{3,+}$, which is given by 
\[
f_{3,+}(x)=x_1^3+x_2^3+x_3^3-x_1^2x_2-x_1^2x_3-x_2^2x_1-x_2^2x_3-x_3^2x_1-x_3^2x_2 +2x_1x_2x_3. 
\]
This is computed in \cite{CW23} and Appendix \ref{sec:app}. 
By \eqref{eq:j1j2j3} and a direct calculation, we see that
\begin{equation}\label{eq:f_{3,+}}
f_{3,+}(a_j)\equiv (a_{j1}+a_{j2}+a_{j3})^3\equiv 1 \mod 2.    
\end{equation}
Suppose that $\Delta_E$ is a prime number. 
Then, $h_E$ is odd.
Hence, we obtain $\mathfrak{P}_E(f_{3,+})\equiv 1\mod 2$, which means $\mathfrak{P}_E(f_{3,+})\neq 0$. 
Since $f_{3,+}$ is a Hecke eigenform by $\cH_{3,+}^\Gamma=\langle f_{3,+}\rangle$, we have $L(\tfrac{1}{2},f_{3,+},\eta_E)\neq 0$ if $\Delta_E\equiv 5 \mod 8$ and $\Delta_E$ is a prime number. 
Note that we can replace the condition $\Delta_E\equiv 5 \mod 8$ by $\eps(\tfrac{1}{2},f_{3,+},\eta_E)=1$, see \eqref{eq:epscong}. 
This non-vanishing theorem was proved in \cite{CW23}*{\S 5}, and several phenomena similar to \eqref{eq:f_{3,+}} were observed there.
Our purpose is to confirm similar phenomena by explicitly computing algebraic modular forms of high degree using our algorithm in \S\ref{sec:3}.

Let us discuss the higher degrees $l$ below.
Unlike the above case $l=3$, Masataka Chida pointed out to us that in the case of even degree $l$, it is sufficient to directly compare the norm $\Nm(x)^{l/2}$ with the algebraic modular forms in order to obtain a congruence relation similar to \eqref{eq:f_{3,+}}. 
So, let us consider only the case that $l$ is even. 
A first example is the case $l=4$.
The space $\cH_{4,+}^\Gamma$ is generated by a single polynomial $F_4=3\, f_{4,+}$, which is given as
\begin{multline*}
F_4(x)=3(x_1^4+x_2^4+x_3^4)+12(x_1^2x_2x_3+x_1x_2^2x_3+x_1x_2x_3^2) \\
-4(x_1^3x_2+x_1^3x_3+x_2^3x_1+x_2^3x_3+x_3^3x_1+x_3^3x_2)-6(x_1^2x_2^2+x_1^2x_3^2+x_2^2x_3^2),
\end{multline*}
cf. Appendix \ref{sec:app}. 
By a direct calculation, we obtain
\begin{equation}\label{eq:diffl4}
3\, F_4(x)-\Nm(x)^2 = 40\times \{(x_1^2x_2x_3+x_1x_2^2x_3+x_1x_2x_3^2) - (x_1^2x_2^2+x_1^2x_3^2+x_2^2x_3^2)\}.  \end{equation}
Hence, $F_4(x)\equiv \Nm(x)^2 \mod 2$, and then we obviously get $F_4(a_j)\equiv \Nm(a_j)^2 \equiv 1 \mod 2$ if $\Delta_E\equiv 5 \mod 8$. 
Using the database \cite{OWY24} and the same argument as \eqref{eq:diffl4}, we obtain the following theorem. 
\begin{thm}\label{thm:congr}
    Assume that $4\le l \le 40$ and $l$ is even.
    Then, there exists an algebraic modular form $f_{l,+}\in \cH_{l,+}^\Gamma$ with integer coefficients so that $f_{l,+}(x)\equiv \Nm(x)^{l/2} \mod 2$. 
\end{thm}
\begin{rem}
\begin{itemize}
    \item We expect that this assertion holds for any even degree $l$. 
    \item Not all algebraic modular forms satisfy the congruence relation with $\Nm(x)^{l/2}$. 
    For example, it follows from \eqref{eq:j1j2j3} that $f^{(2)}_{12,+}(a_j)\equiv 0$ when $\Delta_E\equiv 5 \mod 8$ (see Appendix \ref{sec:app} for $f^{(2)}_{12,+}$).
\end{itemize}
\end{rem}
\begin{proof}
    We write $f^{(1)}_{l,+}$, $f^{(2)}_{l,+}, \dots, f^{(d)}_{l,+}$ $(d=\dim\cH^\Gamma_{l,+})$ for the basis of $\cH^\Gamma_{l,+}$ in the database \cite{OWY24} in order from the front. 
    Except $l=30$ and $38$, by a computer calculation we can check that there exists an odd integer $c$ such that all coefficients of $c\, f^{(1)}_{l,+}(x)$ are integers and the relation $f^{(1)}_{l,+}(x) \equiv \Nm(x)^{l/2}$ holds. 
    As for the cases $l=30$ and $38$, by a computer calculation, we can see that all coefficients of $(f^{(1)}_{30,+}(x)-f^{(3)}_{30,+}(x))/2$ and $(f^{(1)}_{38,+}(x)-3\, f^{(3)}_{38,+}(x))/2$ are integers, and we have the relations
    \[
    \frac{f^{(1)}_{30,+}(x)-f^{(3)}_{30,+}(x)}{2}\equiv \Nm(x)^{15} \mod 2, \qquad \frac{f^{(1)}_{38,+}(x)-3\, f^{(3)}_{38,+}(x)}{2}\equiv \Nm(x)^{19} \mod 2.
    \]
    Therefore, the proof is completed. 
\end{proof}

By using Theorem \ref{thm:congr}, we obtain the following. 
\begin{thm}\label{thm:nonvani}
    Suppose that $4\le l \le 40$, $l$ is even, $f$ is a Hecke eigenform in $\cH_{l,+}^\Gamma\otimes\mathbb{F}_{l,+}$, and $\Delta_E$ is a prime number. 
    If $\eps(\tfrac{1}{2},f,\eta_E)=1$, then we have $L(\frac12,f,\eta_E)\neq 0$.  
\end{thm}
\begin{rem}
    In the following proof of Theorem \ref{thm:nonvani}, we use the fact that the number of Galois orbits in $S_{2l+2}^{\new,+}(\Gamma_0(2))$ is $1$ for every $l$ $(4\le l \le 40)$.  
    A generalization of Maeda's conjecture suggests that, for sufficiently large $l$, it should be always satisfied, see \cite{Tsa14}*{Conjecture 2.4}. 
\end{rem}
\begin{proof}
By Theorem \ref{thm:congr}, there is an algebraic modular form $f_{l,+}$ in $\cH_{l,+}^\Gamma$ such that $f_{l,+}(x)\equiv \Nm(x) \mod 2$. 
Since $h_E$ is odd and $\Nm(a_j)\equiv 1\mod 2$, we have $\mathfrak{P}_E(f_{l,+})\equiv 1\mod 2$, hence we otain $\mathfrak{P}_E(f_{l,+})\neq 0$. 
Since the Galois orbits in $S_{2l+2}^{\new,+}(\Gamma_0(2))$ agree with those in $\cH_{l,+}^\Gamma\otimes\mathbb{F}_{l,+}$, see \S\ref{sec:GOJL}, we can confirm that, for each even integer $l$ $(4\le l\le 40)$, the space $\cH_{l,+}^\Gamma\otimes\mathbb{F}_{l,+}$ is spanned by a single Galois orbit of a Hecke eigenfrom by using LMFDB \cite{lmfdb}. 
For each weight, a link to the site of the Galois orbit is pasted below: 
\[
\text{Weights $2l+2=$ \href{http://www.lmfdb.org/ModularForm/GL2/Q/holomorphic/2/10/a/a/}{$10$}, \href{http://www.lmfdb.org/ModularForm/GL2/Q/holomorphic/2/14/a/b/}{$14$}, \href{http://www.lmfdb.org/ModularForm/GL2/Q/holomorphic/2/18/a/a/}{$18$}, \href{http://www.lmfdb.org/ModularForm/GL2/Q/holomorphic/2/22/a/b/}{$22$}, \href{http://www.lmfdb.org/ModularForm/GL2/Q/holomorphic/2/26/a/b/}{$26$}, \href{http://www.lmfdb.org/ModularForm/GL2/Q/holomorphic/2/30/a/b/}{$30$}, \href{http://www.lmfdb.org/ModularForm/GL2/Q/holomorphic/2/34/a/b/}{$34$}, \href{http://www.lmfdb.org/ModularForm/GL2/Q/holomorphic/2/38/a/b/}{$38$}, \href{http://www.lmfdb.org/ModularForm/GL2/Q/holomorphic/2/42/a/b/}{$42$}, \href{http://www.lmfdb.org/ModularForm/GL2/Q/holomorphic/2/46/a/b/}{$46$}, \href{http://www.lmfdb.org/ModularForm/GL2/Q/holomorphic/2/50/a/b/}{$50$}, \href{http://www.lmfdb.org/ModularForm/GL2/Q/holomorphic/2/54/a/b/}{$54$}, \href{http://www.lmfdb.org/ModularForm/GL2/Q/holomorphic/2/58/a/b/}{$58$}, \href{http://www.lmfdb.org/ModularForm/GL2/Q/holomorphic/2/62/a/b/}{$62$}, \href{http://www.lmfdb.org/ModularForm/GL2/Q/holomorphic/2/66/a/b/}{$66$}, \href{http://www.lmfdb.org/ModularForm/GL2/Q/holomorphic/2/70/a/b/}{$70$}, \href{http://www.lmfdb.org/ModularForm/GL2/Q/holomorphic/2/74/a/b/}{$74$}, \href{http://www.lmfdb.org/ModularForm/GL2/Q/holomorphic/2/78/a/b/}{$78$}, \href{http://www.lmfdb.org/ModularForm/GL2/Q/holomorphic/2/82/a/b/}{$82$}.}
\]
Hence, we can apply the argument in \cite{SWY22}*{Proof of Proposition 4.3}, and then for each Hecke eigenform $f\in\cH_{l,+}^\Gamma\otimes \mathbb{F}_{l,+}$, we obtain $\mathfrak{P}_E(f)\neq 0$. 
Since $\eps(\tfrac{1}{2},f,\eta_E)=1$ is equivalent to $\Delta_E\equiv 5\mod 8$ by \eqref{eq:epscong}, the proof is completed.
\end{proof}

\section{Factorization}\label{sec:fac}

In Appendix \ref{sec:app}, we explicitly describe algebraic modular forms $f_{l,\pm}$, which form a basis of $\cH^\Gamma_{l,\pm}$ $(3\le l\le 12)$. 
Based on our computational results (cf. Appendix \ref{sec:app} and \cite{OWY24}), we find the factrization of $f_{3,+}$, $f_{6,-}$, $f_{9,-}$ as
\[
f_{3,+}(x)=-(-x_1 + x_2 + x_3)(x_1 - x_2 + x_3)(x_1 + x_2 - x_3),
\]
\[
f_{6,-}(x)=x_1x_2x_3(x_1 - x_2)(x_2 - x_3)(x_3 - x_1), \qquad  f_{9,-}=f_{3,+} \, f_{6,-}, 
\]
and the following two predictions, which we prove here. 
\begin{prop}\label{prop:div}
\begin{itemize}
    \item[(1)] When $l$ odd, arbitrary polynomial in $\cH^\Gamma_l$ is divided by $f_{3,+}$.   
    \item[(2)] Arbitrary polynomial in $\cH^\Gamma_{l,-}$ is divided by $f_{6,-}$. 
\end{itemize}    
\end{prop}
\begin{rem}
Note that the space $\bigoplus_{l=0}^\infty \cH^\Gamma_l$ is not closed about the product unlike the ring of holomorphic modular forms, hence the assertions (1) and (2) are not obvious.  
In fact, the relation
\[
f_{7,+}(x)=f_{3,+}(x) \, \left\{ f_{4,+}(x) +\tfrac{16}{3}( - x_1^2 x_2^2 + x_1^2 x_2 x_3 - x_1^2 x_3^2 + x_1 x_2^2 x_3 +  x_1 x_2 x_3^2 - x_2^2 x_3^2) \right\} 
\]
means that $f_{3,+} f_{4,+}$ is not harmonic, namely, $f_{3,+} f_{4,+}\notin\cH^\Gamma_7$. 
We also note that $f_{4,+}$, $f_{6,+}$, $f_{8,+}$, $f_{10,+}$ are irreducible unlike $f_{3,+}$ and $f_{6,-}$.     
\end{rem}
\begin{proof}
Set $\cA_l^\Gamma\coloneqq \Im(\mathrm{Ave}_l)$ and $\cA_{l,\pm}^\Gamma\coloneqq \Im(\mathrm{Ave}_{l,\pm})$. 
Since $\Delta_Q$ and the action of $G$ are commutative, by \cite{Far08}*{Theorem 9.3.1} we have the direct sum decomposition
\begin{align}\label{eq:dsd}
\cA_l^\Gamma=&\bigoplus_{0\le 2m \le l} \Nm(x)^m \, \cH^\Gamma_{l-2m},    \\
\cA_{l,\pm}^\Gamma=&\bigoplus_{0\le 2m \le l} \Nm(x)^m \, \cH^\Gamma_{l-2m,\pm}.  \label{eq:dsd2}  
\end{align}

Let us consider the proof of the assertion (1). 
Write $\mathrm{Proj}_{\cH_l^\Gamma}$ (resp. $\mathrm{Proj}_{\cH_l^\Gamma}'$) for the projection from $\cA_l^\Gamma$ to $\cH_l^\Gamma$ (resp. $\bigoplus_{1\le 2m \le l} \Nm(x)^m \, \cH^\Gamma_{l-2m}$) by \eqref{eq:dsd}. 
A linear mapping $\mathscr{F}_3\colon \cH_{2m}^\Gamma\to \cH_{2m+3}^\Gamma$ is defined by
\[
\mathscr{F}_3(\phi)\coloneqq \mathrm{Proj}_{\cH_{2m+3}^\Gamma}(f_{3,+}\, \phi) , \qquad  \phi\in \cH_{2m}^\Gamma.
\]

Suppose that $\phi\in \cH_{2m}^\Gamma$ and $\mathrm{Proj}_{\cH_{2m+3}^\Gamma}(f_{3,+}\, \phi)=0$. 
Then, $\Nm(x)$ divides $f_{3,+}(x)\, \phi(x)$ by \eqref{eq:dsd}, hence $\Nm(x)$ divides $\phi(x)$ since the ring of polynomials of $x$ over $\Q$ is a unique factorization domain.
This implies $\phi=0$ by $\phi\in \cH_{2m}^\Gamma$ and \eqref{eq:dsd}.  
Therefore, $\mathscr{F}_3$ is injective. 
From this and \eqref{eq:dim1} we find that $\mathscr{F}_3$ is isomorphism.  
Hence, (1) is equivalent to (3) any polynomial in $\mathrm{Im}(\mathscr{F}_3)$ being divisible by $f_{3,+}$. 
Furthermore, (3) is equivalent to (4) $\mathrm{Proj}_{\cH_{2m+3}^\Gamma}'(f_{3,+}\, \phi)$ being divided by $f_{3,+}$ for any $\phi\in\cH^\Gamma_{2m}$.  
If the assertion (1) for $\cH^\Gamma_{2j-1}$ holds for an arbitrary $j$ $(1\le j\le m+1)$, then (4) follows from the definition of $\mathrm{Proj}_{\cH_{2m+3}^\Gamma}'$. 
Thus, we obtain the assertion (1) by the induction for $m$.

Let us consider the proof of the assertion (2). 
Write $\mathrm{Proj}_{\cH_{l,-}^\Gamma}$ for the projection from $\cA_{l,-}^\Gamma$ to $\cH_{l,-}^\Gamma$ by \eqref{eq:dsd2}. 
We use the mapping $\mathscr{F}_6\colon \cH_{l,+}\to \cH_{l+6,-}$ defined by
\[
\mathscr{F}_6(\phi)\coloneqq \mathrm{Proj}_{\cH_{l+6,-}^\Gamma}(f_{6,-}\, \phi) , \qquad  \phi\in \cH_{l,+}^\Gamma.
\]
By the same argument as above, one can prove that $\mathscr{F}_6$ is isomorphic by using \eqref{eq:dim2}. 
Therefore, we can obtain (2) by the same argument as in the proof of (1). 
\end{proof}

\section{Simpler expression by symmetric polynomials}\label{sec:refine}

In studying congruence relations of values of algebraic modular forms at CM points, it was important in the paper \cite{CW23} to take a basis (like $\bb_1$, $\bb_2$, $\bb_3$) of the lattice $L(\cO)$ to ensure the values are integers (cf. \S\ref{sec:4}). 
However, there is no problem in taking another basis of $V_D$ over $\Q$ instead of a basis of $L(\cO)$ to explicitly compute algebraic modular forms. 
In this special case, we have found a simpler expression of algebraic modular forms by another basis and symmetric polynomials. From this expression we find another algorithm for basis of $\cH^\Gamma_{l,\pm}$, which is faster than that of \S\ref{sec:3}. 
Note that the algorithm in \S\ref{sec:3} is more versatile and can be applied to other quaternion algebras.

Take a basis $i$, $j$, $ij$ of $V_D$. 
By the identification $y_1 i + y_2 j + y_3 ij = x_1 \bb_1 + x_1 \bb_2 + x_1 \bb_3$, we obtain a linear transformation
\begin{equation}\label{xy}
\begin{cases}
y_1 &= -x_1+x_2+x_3,\\
y_2 &= x_1-x_2+x_3,\\
y_3 &= x_1+x_2-x_3,
\end{cases}
\qquad\qquad
\begin{cases}
x_1 &=(y_2+y_3)/2,\\
x_2 &=(y_1+y_3)/2,\\
x_3 &=(y_1+y_2)/2. 
\end{cases}
\end{equation}
By $\Nm(y_1i+y_2j+y_3ij)=y_1^2+y_2^2+y_3^2 = y \t y$, $y=\begin{pmatrix}
    y_1&y_2&y_3
\end{pmatrix}$, the Laplace operator $\Delta_Y$ is given by
\[
 \Delta_Y 
= \frac{\partial^2}{\partial y_1^2} 
+ \frac{\partial^2}{\partial y_2^2} 
+ \frac{\partial^2}{\partial y_3^2}.
\]
In addition, a homomorphism $\cF_Y\colon D^\times\to \SO_3$ is defined by
\[
\cF_Y(g)=\gamma_Y \, \cF(g) \, \gamma_Y^{-1} ,\qquad \gamma_Y=\frac{1}{2} \begin{pmatrix}
    0&1&1 \\ 1&0&1 \\ 1&1&0
\end{pmatrix}
\]
where $\SO_3\coloneqq \{ g\in\SL_3(\Q)\mid g\t g=1_3\}$ ($1_3$ is the unit matrix of degree $3$). 
Then, $\cF_Y$ satisfies the same property as $\cF$ with respect to $y=(y_1,y_2,y_3)$. 
In particular, we see
\begin{equation}\label{eq:rep}
    \cF_Y(i)=\begin{pmatrix}
        1&0&0 \\ 0&-1&0 \\ 0&0&-1
    \end{pmatrix}, \quad \cF_Y(w)=\cF(w)=\mathfrak{w} , \quad  \cF_Y(i-j)=\cF(i-j)=\gamma_2.
\end{equation}

For any polynomial $f(x)$ for the variable $x=(x_1,x_2,x_3)$, by \eqref{xy} we identify $f(x)$ with the polynomial
\[
f(y)\coloneqq f( (y_2+y_3)/2, (y_1+y_3)/2,(y_1+y_2)/2)
\]
for the variables $y=(y_1,y_2,y_3)$, that is, $\Q[x_1,x_2,x_3]$ and $\Q[y_1,y_2,y_3]$ are identified by this correspondence. 
For example we have
\[
f_{3,+}(y)=-y_1 y_2 y_3, \quad  f_{6,-}(y)=(y_1^2-y_2^2)(y_1^2-y_3^2)(y_2^2-y_3^2).
\]
By this identification, $\cA_l$, $\cH_l$ and $\cH_l^\Gamma$ are regarded as spaces consisting of homogeneous polynomials of $y=(y_1,y_2,y_3)$.  
In the coordinate $y=(y_1,y_2,y_3)$, $\Gamma$ is identified with the group generated by $\cF_Y(i)$ and $\mathfrak{w}$ in $\SO_3$, and the Atkin-Lehner involution $T_2$ agrees with the action of $\gamma_2$, see \eqref{eq:rep}.

\newcommand{\InvRing}{\Q[y_1^2,y_2^2,y_3^2]^{\mathfrak{S}_3}}

We denote by $\InvRing$
the $\Q$-algebra of symmetric polynomials with variables $(y_1^2,y_2^2,y_3^2)$.
For non-negative even integer $l$, let $\InvRing_l$ denote 
the subspace consisting of homogeneous of degree $l$.
\begin{lem}\label{lemma1}
For non-negative even integer $m$,
\[
\mathcal{H}^\Gamma_{m+3\varepsilon_1+6 \varepsilon_2,\pm} 
= \mathrm{Ker}( \Delta_Y) \cap (f_{3,+}^{\varepsilon_1} f_{6,-}^{\varepsilon_2} \times \InvRing_{m}),
\]
where [$\varepsilon_1 = 0$ or $1$] and [$\varepsilon_2=0$ for plus, $\varepsilon_2=1$ for minus].
\end{lem}
\begin{proof}
    It follows from Lemma \ref{lem:pm} and \eqref{eq:rep} that $\cH_{m,+}^\Gamma=\mathrm{Ker}( \Delta_Y) \cap \InvRing_m$ holds for any even integer $m>0$. 
    Hence, by Proposition \ref{prop:div}, $\mathcal{H}^\Gamma_{m+3\varepsilon_1+6 \varepsilon_2,\pm}$ is included in $f_{3,+}^{\varepsilon_1} f_{6,-}^{\varepsilon_2} \times \InvRing_{m}$. 
    On the other hand, it is obvious that any element in $\mathrm{Ker}( \Delta_Y) \cap (f_{3,+}^{\varepsilon_1} f_{6,-}^{\varepsilon_2} \times \InvRing_{m})$ belongs to $\mathcal{H}^\Gamma_{m+3\varepsilon_1+6 \varepsilon_2,\pm}$.  
\end{proof}

\newcommand{\el}{\mathfrak{e}}

There are several choices of generators of the ring $\InvRing$.
The elementary symmetric polynomials
\[
\el_1(y)\coloneqq y_1^2+y_2^2+y_3^2,
\qquad
\el_2(y)\coloneqq y_1^2 y_2^2+y_1^2 y_3^2 + y_2^2 y_3^2,
\qquad
\el_3(y)\coloneqq y_1^2 y_2^2 y_3^2
\]
are one of the choice, so $\InvRing = \Q[\el_1,\el_2,\el_3]$.

Now we define the following second order linear partial differential operators
\begin{align*}
\Delta_{E} &\coloneqq 4 \el_1 \frac{\partial^2}{\partial \el_1^2}
+ 4(\el_1 \el_2+3 \el_3) \frac{\partial^2}{\partial \el_2^2}
+ 4 \el_2 \el_3 \frac{\partial^2}{\partial \el_3^2}
+ 16 \el_2 \frac{\partial^2}{\partial \el_1 \partial \el_2}
+ 24 \el_3 \frac{\partial^2}{\partial \el_1 \partial \el_3}
+ 16 \el_1 \el_3 \frac{\partial^2}{\partial \el_2 \partial \el_3},
\\
\Delta_{\varepsilon_1,\varepsilon_2} &\coloneqq \Delta_{E} 
+ (1+2\varepsilon_1) \left(6 \frac{\partial}{\partial \el_1} 
+ 4 \el_1 \frac{\partial}{\partial \el_2} 
+ 2 \el_2 \frac{\partial}{\partial \el_3}\right)
+ \varepsilon_2 \left(
  24 \frac{\partial}{\partial \el_1} 
+ 8 \el_1 \frac{\partial}{\partial \el_2} \right)
\end{align*}
for $\varepsilon_1$, $\varepsilon_2 \in \{ 0$, $1\}$.

\begin{lem}
These operators satisfy the following commutative diagram:
\[
\begin{array}{ccccc}
\Q[\el_1,\el_2,\el_3] &\overset\sim\longrightarrow &
f_{3,+}^{\varepsilon_1} f_{6,-}^{\varepsilon_2} \InvRing 
&\subset& \Q[y_1,y_2,y_3]
\\
\Delta_{\varepsilon_1,\varepsilon_2}\downarrow
& & \downarrow && \downarrow \Delta_Y 
\\
\Q[\el_1,\el_2,\el_3] &\overset\sim\longrightarrow &f_{3,+}^{\varepsilon_1} f_{6,-}^{\varepsilon_2} \InvRing 
&\subset& \Q[y_1,y_2,y_3]
\end{array}
\]
\end{lem}
\begin{proof}
This can be proved by a direct calculation. 
\end{proof}

Let
$\Q[\el_1,\el_2,\el_3]_m$ be the homogeneous part of degree $m$, where
the degrees of $\el_1$, $\el_2$, $\el_3$ are $2$, $4$, $6$ respectively.
We define
\begin{align*}
\eH{m}{\varepsilon_1,\varepsilon_2}
&\coloneqq \mathrm{Ker}( \Delta_{\varepsilon_1,\varepsilon_2}) \cap
\Q[\el_1,\el_2,\el_3]_m
\end{align*}
for non-negative even integer $m$.
We rephrase Lemma~\ref{lemma1} with these notations.
\begin{lem}
For non-negative even integer $m$,
\[
\mathcal{H}^\Gamma_{m+ 3 \varepsilon_1+6 \varepsilon_2,\pm} 
= f_{3,+}^{\varepsilon_1} f_{6,-}^{\varepsilon_2} \times 
\eH{m}{\varepsilon_1,\varepsilon_2}
\]
where [$\varepsilon_1 = 0$ or $1$] 
and [$\varepsilon_2=0$ for plus, $\varepsilon_2=1$ for minus].
\end{lem}

So, once we obtain a basis of 
$\eH{m}{\varepsilon_1,\varepsilon_2}$
then 
multiplying $f_{3,+}$, $f_{6,-}$ if necessary,
and change the coordinates from $(\el_1,\el_2,\el_3)$ to $(y_1,y_2,y_3)$ and then to $(x_1,x_2,x_3)$,
we obtain the basis of $\mathcal{H}_{l,\pm}^\Gamma$.

\medskip

\noindent
{\bf Algorithm}:
Put $M_m\coloneqq
\{ \el_1^{j_1} \el_2^{j_2} \el_3^{j_3} \mid
0 \leq j_3 \leq m/6$, 
$0 \leq j_2 \leq (m-6j_3)/4$, $j_1 = (m-6j_3-4j_2)/2 \}$
the set of monomials with weighted degree $m$.
We see that $\Q[\el_1,\el_2,\el_3]_m$ is a vector space 
with basis $M_m$.
Consider a linear combination $\tilde{f}$ of elements in $M_m$.
Look at the coefficients of a polynomial 
$\Delta_{\varepsilon_1,\varepsilon_2} \tilde{f} \in \Q[\el_1,\el_2,\el_3]_{m-2}$
with basis $M_{m-2}$.
Solve the system of homogeneous linear equations
 so that all the coefficients above are zero. 
For each solution, we obtain an element in 
$\eH{m}{\varepsilon_1,\varepsilon_2}$,
and then we obtain a basis of 
$\eH{m}{\varepsilon_1,\varepsilon_2}$.
This algorithm works by the same argument as in \S\ref{sec:3} for $\mathrm{Ker}(\Delta_l)$.

\appendix
\section{Computational results for the variables \texorpdfstring{$(x_1,x_2,x_3)$}{TEXT}}\label{sec:app}

In what follows, we present numerical examples for bases of $\cH_{l,\pm}^\Gamma$ $(3\le l \le 12)$ by inputting the algorithm of \S\ref{sec:3} into \textsf{Magma} \cite{BCP97}. 
See \cite{OWY24} for data on higher degrees. 

\vspace{5mm}

\noindent
$l=3$ : $\dim \mathcal{H}_{3,+}^{\Gamma}=1$, $\dim \mathcal{H}_{3,-}^{\Gamma}=0$\\
$\mathcal{H}_{3,+}^{\Gamma} = \left< f_{3,+} \right>$

\begin{align}
\begin{autobreak}
f_{3,+} = 
x_1^3 - x_1^2x_2 - x_1^2x_3 - x_1x_2^2 + 2x_1x_2x_3 - x_1x_3^2 + x_2^3 - x_2^2x_3 -x_2x_3^2 + x_3^3 \nonumber
\end{autobreak}
\end{align}

\vspace{5mm}

\noindent
$l=4$ : $\dim \mathcal{H}_{4,+}^{\Gamma}=1$, $\dim \mathcal{H}_{4,-}^{\Gamma}=0$\\
$\mathcal{H}_{4,+}^{\Gamma} = \left< f_{4,+} \right>$

\begin{align}
\begin{autobreak}
f_{4,+} =
x_1^4 
- \frac{4}{3}x_1^3x_2 
- \frac{4}{3}x_1^3x_3 
- 2x_1^2x_2^2 
+ 4x_1^2x_2x_3 
- 2x_1^2x_3^2 
- \frac{4}{3}x_1x_2^3 
+ 4x_1x_2^2x_3 
+ 4x_1x_2x_3^2 
- \frac{4}{3}x_1x_3^3 
+ x_2^4 
- \frac{4}{3}x_2^3x_3 
- 2x_2^2x_3^2 
- \frac{4}{3}x_2x_3^3 
+ x_3^4 \nonumber
\end{autobreak}
\end{align}

\vspace{5mm}

\noindent
$l=5$ : $\dim \mathcal{H}_{5,+}^{\Gamma}=0$, $\dim \mathcal{H}_{5,-}^{\Gamma}=0$

\vspace{5mm}

\noindent
$l=6$ : $\dim \mathcal{H}_{6,+}^{\Gamma}=1$, $\dim \mathcal{H}_{6,-}^{\Gamma}=1$\\
$\mathcal{H}_{6,+}^{\Gamma} = \left< f_{6,+} \right>$, $\mathcal{H}_{6,-}^{\Gamma} = \left< f_{6,-} \right>$

\begin{align}
\begin{autobreak}
f_{6,+}=
3x_1^6
- 6x_1^5x_2 
- 6x_1^5x_3
- 15x_1^4x_2^2 
+ 30x_1^4x_2x_3 
- 15x_1^4x_3^2
+ 40x_1^3x_2^3 
- 30x_1^3x_2^2x_3 
- 30x_1^3x_2x_3^2 
+ 40x_1^3x_3^3 
-15x_1^2x_2^4 
- 30x_1^2x_2^3x_3 
+ 90x_1^2x_2^2x_3^2 
- 30x_1^2x_2x_3^3 
-15x_1^2x_3^4 
- 6x_1x_2^5 
+ 30x_1x_2^4x_3 
- 30x_1x_2^3x_3^2 
- 30x_1x_2^2x_3^3 
+ 30x_1x_2x_3^4 
- 6x_1x_3^5 
+ 3x_2^6 
- 6x_2^5x_3 
-15x_2^4x_3^2 
+ 40x_2^3x_3^3 
- 15x_2^2x_3^4 
- 6x_2x_3^5 
+ 3x_3^6 \nonumber 
\end{autobreak}
\end{align}
\begin{align}
\begin{autobreak}
f_{6,-}=
-x_1^3x_2^2x_3 
+ x_1^3x_2x_3^2 
+ x_1^2x_2^3x_3 
- x_1^2x_2x_3^3 
- x_1x_2^3x_3^2 
+x_1x_2^2x_3^3 \nonumber
\end{autobreak}
\end{align}

\vspace{5mm}

\noindent
$l=7$ : $\dim \mathcal{H}_{7,+}^{\Gamma}=1$, $\dim \mathcal{H}_{7,-}^{\Gamma}=0$\\
$\mathcal{H}_{7,+}^{\Gamma} = \left< f_{7,+} \right>$

\begin{align}
\begin{autobreak}
f_{7,+}=
x_1^7 
- \frac{7}{3}x_1^6x_2 
- \frac{7}{3}x_1^6x_3 
- 7x_1^5x_2^2 
+ 14x_1^5x_2x_3 
- 7x_1^5x_3^2
+ \frac{25}{3}x_1^4x_2^3 
+ 5x_1^4x_2^2x_3
 + 5x_1^4x_2x_3^2 
 + \frac{25}{3}x_1^4x_3^3 
+ \frac{25}{3}x_1^3x_2^4 
 - \frac{100}{3}x_1^3x_2^3x_3 
 + 10x_1^3x_2^2x_3^2 
- \frac{100}{3}x_1^3x_2x_3^3 
 + \frac{25}{3}x_1^3x_3^4 
 - 7x_1^2x_2^5 
 + 5x_1^2x_2^4x_3 
+ 10x_1^2x_2^3x_3^2 
 + 10x_1^2x_2^2x_3^3 
 + 5x_1^2x_2x_3^4 
 - 7x_1^2x_3^5 - 
 \frac{7}{3}x_1x_2^6 
 + 14x_1x_2^5x_3 
 + 5x_1x_2^4x_3^2 
 - \frac{100}{3}x_1x_2^3x_3^3 
+5x_1x_2^2x_3^4 
+ 14x_1x_2x_3^5
- \frac{7}{3}x_1x_3^6
 + x_2^7
  - \frac{7}{3}x_2^6x_3
-7x_2^5x_3^2 
+ \frac{25}{3}x_2^4x_3^3
 + \frac{25}{3}x_2^3x_3^4
  - 7x_2^2x_3^5 
-\frac{7}{3}x_2x_3^6 
+ x_3^7 \nonumber
\end{autobreak}
\end{align}

\vspace{5mm}

\noindent
$l=8$ : $\dim \mathcal{H}_{8,+}^{\Gamma}=1$, $\dim \mathcal{H}_{8,-}^{\Gamma}=0$\\
$\mathcal{H}_{8,+}^{\Gamma} = \left< f_{8,+} \right>$

\begin{align}
\begin{autobreak}
f_{8,+}=
3x_1^8
 - 8x_1^7x_2
 - 8x_1^7x_3
 - 28x_1^6x_2^2 
+ 56x_1^6x_2x_3
 - 28x_1^6x_3^2
 - 56x_1^5x_2^3 
 + 168x_1^5x_2^2x_3 
 + 168x_1^5x_2x_3^2 
 - 56x_1^5x_3^3 
+  210x_1^4x_2^4
   - 280x_1^4x_2^3x_3
    - 420x_1^4x_2^2x_3^2 
    - 280x_1^4x_2x_3^3
+ 210x_1^4x_3^4 
- 56x_1^3x_2^5 
- 280x_1^3x_2^4x_3 
+ 560x_1^3x_2^3x_3^2
+ 560x_1^3x_2^2x_3^3 
 - 280x_1^3x_2x_3^4 
 - 56x_1^3x_3^5 
 - 28x_1^2x_2^6 + 
 168x_1^2x_2^5x_3 
 - 420x_1^2x_2^4x_3^2 
 + 560x_1^2x_2^3x_3^3
- 420x_1^2x_2^2x_3^4 
 + 168x_1^2x_2x_3^5 
 - 28x_1^2x_3^6 
 - 8x_1x_2^7 + 
 56x_1x_2^6x_3 
 + 168x_1x_2^5x_3^2 
 - 280x_1x_2^4x_3^3 
 - 280x_1x_2^3x_3^4
+168x_1x_2^2x_3^5 
+ 56x_1x_2x_3^6 
- 8x_1x_3^7 
+ 3x_2^8 
- 8x_2^7x_3 
- 28x_2^6x_3^2 
 - 56x_2^5x_3^3 
 + 210x_2^4x_3^4 
 - 56x_2^3x_3^5 
- 28x_2^2x_3^6
  - 8x_2x_3^7 
  + 3x_3^8 \nonumber
\end{autobreak}
\end{align}

\vspace{5mm}

\noindent
$l=9$ : $\dim \mathcal{H}_{9,+}^{\Gamma}=1$, $\dim \mathcal{H}_{9,-}^{\Gamma}=1$\\
$\mathcal{H}_{9,+}^{\Gamma} = \left< f_{9,+} \right>$, $\mathcal{H}_{9,-}^{\Gamma} = \left< f_{9,-} \right>$

\begin{align}
\begin{autobreak}
f_{9,+}=
 x_1^9 
 - 3x_1^8x_2
 - 3x_1^8x_3 
 - 12x_1^7x_2^2 
 + 24x_1^7x_2x_3 
 - 12x_1^7x_3^2
+56x_1^6x_2^3
 - 42x_1^6x_2^2x_3 
 - 42x_1^6x_2x_3^2 
 + 56x_1^6x_3^3 
-42x_1^5x_2^4
 - 84x_1^5x_2^3x_3 
 + 252x_1^5x_2^2x_3^2 
 - 84x_1^5x_2x_3^3 
-42x_1^5x_3^4 
- 42x_1^4x_2^5 
+ 210x_1^4x_2^4x_3 
- 210x_1^4x_2^3x_3^2 
-210x_1^4x_2^2x_3^3 
+ 210x_1^4x_2x_3^4 
- 42x_1^4x_3^5 
+ 56x_1^3x_2^6 
-84x_1^3x_2^5x_3 
- 210x_1^3x_2^4x_3^2 
+ 560x_1^3x_2^3x_3^3 
-210x_1^3x_2^2x_3^4 
- 84x_1^3x_2x_3^5 
+ 56x_1^3x_3^6 
- 12x_1^2x_2^7
-42x_1^2x_2^6x_3 
+ 252x_1^2x_2^5x_3^2 
- 210x_1^2x_2^4x_3^3 
-210x_1^2x_2^3x_3^4 
+ 252x_1^2x_2^2x_3^5
 - 42x_1^2x_2x_3^6
 - 12x_1^2x_3^7
- 3x_1x_2^8 
+ 24x_1x_2^7x_3
 - 42x_1x_2^6x_3^2
  - 84x_1x_2^5x_3^3 
+ 210x_1x_2^4x_3^4 
 - 84x_1x_2^3x_3^5 
 - 42x_1x_2^2x_3^6 
 + 24x_1x_2x_3^7 
-3x_1x_3^8 
+ x_2^9 
- 3x_2^8x_3
 - 12x_2^7x_3^2
  + 56x_2^6x_3^3
-42x_2^5x_3^4
 - 42x_2^4x_3^5 
 + 56x_2^3x_3^6
  - 12x_2^2x_3^7
   - 3x_2x_3^8 
+ x_3^9 \nonumber
\end{autobreak}
\end{align}

\begin{align}
\begin{autobreak}
f_{9,-}=
-x_1^6x_2^2x_3 
+ x_1^6x_2x_3^2 
+ 2x_1^5x_2^3x_3 
- 2x_1^5x_2x_3^3 
- 5x_1^4x_2^3x_3^2 
 + 5x_1^4x_2^2x_3^3 
 - 2x_1^3x_2^5x_3 
 + 5x_1^3x_2^4x_3^2 
- 5x_1^3x_2^2x_3^4 
+ 2x_1^3x_2x_3^5 
+ x_1^2x_2^6x_3 
- 5x_1^2x_2^4x_3^3
+ 5x_1^2x_2^3x_3^4 
 - x_1^2x_2x_3^6 
 - x_1x_2^6x_3^2 
 + 2x_1x_2^5x_3^3
-2x_1x_2^3x_3^5 
+ x_1x_2^2x_3^6 \nonumber
\end{autobreak}
\end{align}

\vspace{5mm}

\noindent
$l=10$ : $\dim \mathcal{H}_{10,+}^{\Gamma}=1$, $\dim \mathcal{H}_{10,-}^{\Gamma}=1$\\
$\mathcal{H}_{10,+}^{\Gamma} = \left< f_{10,+} \right>$, $\mathcal{H}_{10,-}^{\Gamma} = \left< f_{10,-} \right>$

\begin{align}
\begin{autobreak}
f_{10,+}=
x_1^{10} 
- \frac{10}{3}x_1^9x_2
 - \frac{10}{3}x_1^9x_3 
 - 15x_1^8x_2^2 
 + 30x_1^8x_2x_3
-15x_1^8x_3^2 
+ 44x_1^7x_2^3
 - 6x_1^7x_2^2x_3
  - 6x_1^7x_2x_3^2 
+44x_1^7x_3^3 
+ 14x_1^6x_2^4
 - 182x_1^6x_2^3x_3 
 + 168x_1^6x_2^2x_3^2 
-182x_1^6x_2x_3^3 
+ 14x_1^6x_3^4 
- 84x_1^5x_2^5 
+ 168x_1^5x_2^4x_3
-42x_1^5x_2^3x_3^2 
- 42x_1^5x_2^2x_3^3 
+ 168x_1^5x_2x_3^4 
- 84x_1^5x_3^5 
+ 14x_1^4x_2^6 
+ 168x_1^4x_2^5x_3 
- 210x_1^4x_2^4x_3^2
+140x_1^4x_2^3x_3^3
 - 210x_1^4x_2^2x_3^4 
 + 168x_1^4x_2x_3^5
+14x_1^4x_3^6 
+ 44x_1^3x_2^7 
- 182x_1^3x_2^6x_3
- 42x_1^3x_2^5x_3^2 
+140x_1^3x_2^4x_3^3 
+ 140x_1^3x_2^3x_3^4 
- 42x_1^3x_2^2x_3^5
-182x_1^3x_2x_3^6 
+ 44x_1^3x_3^7 
- 15x_1^2x_2^8 
- 6x_1^2x_2^7x_3
+168x_1^2x_2^6x_3^2
 - 42x_1^2x_2^5x_3^3
  - 210x_1^2x_2^4x_3^4
-42x_1^2x_2^3x_3^5 
+ 168x_1^2x_2^2x_3^6 
- 6x_1^2x_2x_3^7
 - 15x_1^2x_3^8
- \frac{10}{3}x_1x_2^9 
 + 30x_1x_2^8x_3 
 - 6x_1x_2^7x_3^2 
 - 182x_1x_2^6x_3^3 
+168x_1x_2^5x_3^4 
+ 168x_1x_2^4x_3^5 
- 182x_1x_2^3x_3^6 
- 6x_1x_2^2x_3^7 
+ 30x_1x_2x_3^8 
- \frac{10}{3}x_1x_3^9 
+ x_2^{10} 
- \frac{10}{3}x_2^9x_3 
- 15x_2^8x_3^2
+44x_2^7x_3^3 
+ 14x_2^6x_3^4
 - 84x_2^5x_3^5
  + 14x_2^4x_3^6 
  + 44x_2^3x_3^7
- 15x_2^2x_3^8
 - \frac{10}{3}x_2x_3^9 
 + x_3^{10} \nonumber
\end{autobreak}
\end{align}

\begin{align}
\begin{autobreak}
f_{10,-}=
-x_1^7x_2^2x_3 
+ x_1^7x_2x_3^2 
+ \frac{7}{3}x_1^6x_2^3x_3 
- \frac{7}{3}x_1^6x_2x_3^3 
- 7x_1^5x_2^3x_3^2 
 + 7x_1^5x_2^2x_3^3 
 - \frac{7}{3}x_1^3x_2^6x_3
+7x_1^3x_2^5x_3^2 
- 7x_1^3x_2^2x_3^5
+ \frac{7}{3}x_1^3x_2x_3^6 
+ x_1^2x_2^7x_3 
- 7x_1^2x_2^5x_3^3 
 + 7x_1^2x_2^3x_3^5
  - x_1^2x_2x_3^7
   - x_1x_2^7x_3^2
+ \frac{7}{3}x_1x_2^6x_3^3 
 - \frac{7}{3}x_1x_2^3x_3^6 
 + x_1x_2^2x_3^7 \nonumber
\end{autobreak}
\end{align}

\vspace{5mm}

\noindent
$l=11$ : $\dim \mathcal{H}_{11,+}^{\Gamma}=1$, $\dim \mathcal{H}_{11,-}^{\Gamma}=0$\\
$\mathcal{H}_{11,+}^{\Gamma} = \left< f_{11,+} \right>$

\begin{align}
\begin{autobreak}
f_{11,+}=
3x_1^{11} 
- 11x_1^{10}x_2
 - 11x_1^{10}x_3 
 - 55x_1^9x_2^2 
 + 110x_1^9x_2x_3 
- 55x_1^9x_3^2 
 + 15x_1^8x_2^3 
 + 225x_1^8x_2^2x_3 
 + 225x_1^8x_2x_3^2 
+15x_1^8x_3^3 
+ 510x_1^7x_2^4 
- 1080x_1^7x_2^3x_3 
- 540x_1^7x_2^2x_3^2 
-1080x_1^7x_2x_3^3 
+ 510x_1^7x_3^4 
- 462x_1^6x_2^5 
- 630x_1^6x_2^4x_3 
+2100x_1^6x_2^3x_3^2 
+ 2100x_1^6x_2^2x_3^3 
- 630x_1^6x_2x_3^4 
-462x_1^6x_3^5
 - 462x_1^5x_2^6 
 + 2772x_1^5x_2^5x_3
  - 1890x_1^5x_2^4x_3^2 
 - 840x_1^5x_2^3x_3^3
  - 1890x_1^5x_2^2x_3^4 
  + 2772x_1^5x_2x_3^5 
-462x_1^5x_3^6 
+ 510x_1^4x_2^7 
- 630x_1^4x_2^6x_3 
- 1890x_1^4x_2^5x_3^2
+1050x_1^4x_2^4x_3^3 
+ 1050x_1^4x_2^3x_3^4 
- 1890x_1^4x_2^2x_3^5 
-630x_1^4x_2x_3^6 
+ 510x_1^4x_3^7 
+ 15x_1^3x_2^8 
- 1080x_1^3x_2^7x_3 
+2100x_1^3x_2^6x_3^2
 - 840x_1^3x_2^5x_3^3 
 + 1050x_1^3x_2^4x_3^4 
-840x_1^3x_2^3x_3^5 
+ 2100x_1^3x_2^2x_3^6 
- 1080x_1^3x_2x_3^7 
+15x_1^3x_3^8 
- 55x_1^2x_2^9 
+ 225x_1^2x_2^8x_3 
- 540x_1^2x_2^7x_3^2 
+2100x_1^2x_2^6x_3^3
 - 1890x_1^2x_2^5x_3^4 
 - 1890x_1^2x_2^4x_3^5 
+2100x_1^2x_2^3x_3^6 
- 540x_1^2x_2^2x_3^7 
+ 225x_1^2x_2x_3^8 
-55x_1^2x_3^9 
- 11x_1x_2^{10} 
+ 110x_1x_2^9x_3 
+ 225x_1x_2^8x_3^2 
-1080x_1x_2^7x_3^3
 - 630x_1x_2^6x_3^4
  + 2772x_1x_2^5x_3^5 
- 630x_1x_2^4x_3^6
  - 1080x_1x_2^3x_3^7 
  + 225x_1x_2^2x_3^8 
  + 110x_1x_2x_3^9
 - 11x_1x_3^{10}
  + 3x_2^{11} 
  - 11x_2^{10}x_3
   - 55x_2^9x_3^2 
   + 15x_2^8x_3^3 
+  510x_2^7x_3^4
   - 462x_2^6x_3^5
    - 462x_2^5x_3^6 
    + 510x_2^4x_3^7 
+  15x_2^3x_3^8 
  - 55x_2^2x_3^9 
  - 11x_2x_3^{10} 
  + 3x_3^{11} \nonumber
\end{autobreak}
\end{align}

\vspace{5mm}

\noindent
$l=12$ : $\dim \mathcal{H}_{12,+}^{\Gamma}=2$, $\dim \mathcal{H}_{12,-}^{\Gamma}=1$\\
$\mathcal{H}_{12,+}^{\Gamma} = \left< f_{12,+}^{(1)}, f_{12,+}^{(2)} \right>$, $\mathcal{H}_{12,-}^{\Gamma} = \left< f_{12,-} \right>$

\begin{align}
\begin{autobreak}
f_{12,+}^{(1)}=
x_1^{12} 
- 4x_1^{11}x_2 
- 4x_1^{11}x_3 
- 22x_1^{10}x_2^2 
+ 44x_1^{10}x_2x_3
-22x_1^{10}x_3^2 
+ 110x_1^9x_2^2x_3 
+ 110x_1^9x_2x_3^2
 + 275x_1^8x_2^4 
-550x_1^8x_2^3x_3 
- 330x_1^8x_2^2x_3^2
 - 550x_1^8x_2x_3^3 
 + 275x_1^8x_3^4
- 264x_1^7x_2^5 
- 440x_1^7x_2^4x_3
 + 1320x_1^7x_2^3x_3^2
  + 1320x_1^7x_2^2x_3^3
   - 440x_1^7x_2x_3^4 
   - 264x_1^7x_3^5 
+ 924x_1^6x_2^5x_3 
- 3850x_1^6x_2^4x_3^2 
+ 6160x_1^6x_2^3x_3^3
-3850x_1^6x_2^2x_3^4 
+ 924x_1^6x_2x_3^5
 - 264x_1^5x_2^7
+ 924x_1^5x_2^6x_3 
+ 5544x_1^5x_2^5x_3^2
 - 7700x_1^5x_2^4x_3^3 
- 7700x_1^5x_2^3x_3^4 
+ 5544x_1^5x_2^2x_3^5
 + 924x_1^5x_2x_3^6
-264x_1^5x_3^7
 + 275x_1^4x_2^8
  - 440x_1^4x_2^7x_3
   - 3850x_1^4x_2^6x_3^2
  -7700x_1^4x_2^5x_3^3 
  + 26950x_1^4x_2^4x_3^4 
  - 7700x_1^4x_2^3x_3^5 
- 3850x_1^4x_2^2x_3^6
 - 440x_1^4x_2x_3^7 
 + 275x_1^4x_3^8 
- 550x_1^3x_2^8x_3 
+ 1320x_1^3x_2^7x_3^2 
+ 6160x_1^3x_2^6x_3^3 
-7700x_1^3x_2^5x_3^4
 - 7700x_1^3x_2^4x_3^5
  + 6160x_1^3x_2^3x_3^6 
+ 1320x_1^3x_2^2x_3^7 
- 550x_1^3x_2x_3^8
 - 22x_1^2x_2^{10}
+ 110x_1^2x_2^9x_3 
- 330x_1^2x_2^8x_3^2 
+ 1320x_1^2x_2^7x_3^3 
- 3850x_1^2x_2^6x_3^4 
+ 5544x_1^2x_2^5x_3^5 
- 3850x_1^2x_2^4x_3^6 
+1320x_1^2x_2^3x_3^7
 - 330x_1^2x_2^2x_3^8 
 + 110x_1^2x_2x_3^9 
-22x_1^2x_3^{10} 
- 4x_1x_2^{11} 
+ 44x_1x_2^{10}x_3
 + 110x_1x_2^9x_3^2  
- 550x_1x_2^8x_3^3 
- 440x_1x_2^7x_3^4 
+ 924x_1x_2^6x_3^5 
+924x_1x_2^5x_3^6
 - 440x_1x_2^4x_3^7
  - 550x_1x_2^3x_3^8 
+110x_1x_2^2x_3^9 
+ 44x_1x_2x_3^{10} 
- 4x_1x_3^{11} + x_2^{12}
 - 4x_2^{11}x_3 
- 22x_2^{10}x_3^2
 + 275x_2^8x_3^4
  - 264x_2^7x_3^5 
  - 264x_2^5x_3^7  
 +275x_2^4x_3^8 
 - 22x_2^2x_3^{10}
  - 4x_2x_3^{11} 
  + x_3^{12} \nonumber
\end{autobreak}
\end{align}

\begin{align}
\begin{autobreak}
f_{12,+}^{(2)}=
10x_1^9x_2^3 
- 15x_1^9x_2^2x_3
 - 15x_1^9x_2x_3^2 
 + 10x_1^9x_3^3 
- 30x_1^8x_2^4 
 + 15x_1^8x_2^3x_3 
 + 90x_1^8x_2^2x_3^2 
 + 15x_1^8x_2x_3^3 
-30x_1^8x_3^4
 + 120x_1^7x_2^4x_3 
 - 180x_1^7x_2^3x_3^2 
-180x_1^7x_2^2x_3^3 
+ 120x_1^7x_2x_3^4 
+ 42x_1^6x_2^6
 - 126x_1^6x_2^5x_3 
+ 105x_1^6x_2^4x_3^2 
+ 105x_1^6x_2^2x_3^4 
- 126x_1^6x_2x_3^5 
+42x_1^6x_3^6 
- 126x_1^5x_2^6x_3 
+ 210x_1^5x_2^4x_3^3 
+ 210x_1^5x_2^3x_3^4 
 - 126x_1^5x_2x_3^6 
 - 30x_1^4x_2^8
  + 120x_1^4x_2^7x_3 
+ 105x_1^4x_2^6x_3^2 
+ 210x_1^4x_2^5x_3^3
 - 1050x_1^4x_2^4x_3^4 
+210x_1^4x_2^3x_3^5 
+ 105x_1^4x_2^2x_3^6 
+ 120x_1^4x_2x_3^7 
-30x_1^4x_3^8 
+ 10x_1^3x_2^9 
+ 15x_1^3x_2^8x_3
 - 180x_1^3x_2^7x_3^2 
+210x_1^3x_2^5x_3^4 
+ 210x_1^3x_2^4x_3^5
 - 180x_1^3x_2^2x_3^7 
+15x_1^3x_2x_3^8 
+ 10x_1^3x_3^9
 - 15x_1^2x_2^9x_3
  + 90x_1^2x_2^8x_3^2 
-180x_1^2x_2^7x_3^3 
+ 105x_1^2x_2^6x_3^4 
+ 105x_1^2x_2^4x_3^6 
-180x_1^2x_2^3x_3^7 
+ 90x_1^2x_2^2x_3^8 
- 15x_1^2x_2x_3^9 
-15x_1x_2^9x_3^2 
+ 15x_1x_2^8x_3^3 
+ 120x_1x_2^7x_3^4 
- 126x_1x_2^6x_3^5 
 - 126x_1x_2^5x_3^6 
 + 120x_1x_2^4x_3^7
  + 15x_1x_2^3x_3^8 
-15x_1x_2^2x_3^9 
+ 10x_2^9x_3^3
 - 30x_2^8x_3^4
  + 42x_2^6x_3^6 
-30x_2^4x_3^8 
+ 10x_2^3x_3^9 \nonumber
\end{autobreak}
\end{align}

\begin{align}
\begin{autobreak}
f_{12,-}=
-5x_1^9x_2^2x_3 
+ 5x_1^9x_2x_3^2 
+ 15x_1^8x_2^3x_3 
- 15x_1^8x_2x_3^3 
-60x_1^7x_2^3x_3^2
 + 60x_1^7x_2^2x_3^3
  - 42x_1^6x_2^5x_3 
+105x_1^6x_2^4x_3^2
 - 105x_1^6x_2^2x_3^4 
 + 42x_1^6x_2x_3^5 
+42x_1^5x_2^6x_3
 - 210x_1^5x_2^4x_3^3 
 + 210x_1^5x_2^3x_3^4 
- 42x_1^5x_2x_3^6 
 - 105x_1^4x_2^6x_3^2 
 + 210x_1^4x_2^5x_3^3 
-210x_1^4x_2^3x_3^5 
+ 105x_1^4x_2^2x_3^6 
- 15x_1^3x_2^8x_3 
+60x_1^3x_2^7x_3^2
 - 210x_1^3x_2^5x_3^4
  + 210x_1^3x_2^4x_3^5 
- 60x_1^3x_2^2x_3^7 
 + 15x_1^3x_2x_3^8 
 + 5x_1^2x_2^9x_3 
 - 60x_1^2x_2^7x_3^3
 + 105x_1^2x_2^6x_3^4 
 - 105x_1^2x_2^4x_3^6 
 + 60x_1^2x_2^3x_3^7 
-5x_1^2x_2x_3^9 
- 5x_1x_2^9x_3^2 
+ 15x_1x_2^8x_3^3 
- 42x_1x_2^6x_3^5 
+42x_1x_2^5x_3^6 
- 15x_1x_2^3x_3^8 
+ 5x_1x_2^2x_3^9 \nonumber
\end{autobreak}
\end{align}


\section{Computational results for the variables \texorpdfstring{$(\el_1,\el_2,\el_3)$}{TEXT}}\label{sec:app2}

In this appendix, we list the algebraic modular forms in $\el$ coordinates explained in \S\ref{sec:refine} for low degrees $l \leq 12$.
We also give a comparison of the expression given in Appendix \ref{sec:app} in $x$-coordinates and that given here.
We have seen 
$\dim \eH{m}{\varepsilon_1, \varepsilon_2} 
= k+1$
if $m-12k \in \{0,4,6,8,10,14\}$,
and $\dim \eH{m}{\varepsilon_1, \varepsilon_2} =0$ otherwise.
If $\dim \eH{m}{\varepsilon_1, \varepsilon_2} =1$, then
we will choose $\modF{m}{\varepsilon_1, \varepsilon_2}  \in \eH{m}{\varepsilon_1, \varepsilon_2}$ so that the coefficient of $\el_1^{m/2}$
is one.
Then we determine the explicit constant multiple
between $f_{l,\pm}$ and $\modF{m}{\varepsilon_1, \varepsilon_2}$.

On the other hand,
when the dimension of $\eH{m}{\varepsilon_1, \varepsilon_2}$
is greater than $1$, then we choose a basis 
$\modF{m}{\varepsilon_1, \varepsilon_2}^{(1)}$, 
$\ldots$,
$\modF{m}{\varepsilon_1, \varepsilon_2}^{(d)}$ 
of this space $\eH{m}{\varepsilon_1, \varepsilon_2}$
according to graded reverse lexicographic order, say,  grevlex, 
of monomials in $\el_1,\el_2,\el_3$.
In the case $m=12$,
we have seen $\dim \eH{12}{\varepsilon_1, \varepsilon_2} =2$,
and we will see that 
the basis $f_{12,+}^{(1)}, f_{12,+}^{(2)}$ of $\HG{12,+}$ given in
Appendix~\ref{sec:app} is a linear combination of the basis
$\modF{12}{\varepsilon_1, \varepsilon_2}^{(1)}$,
$\modF{12}{\varepsilon_1, \varepsilon_2}^{(2)}$ of
$\eH{12}{0,0}$ of this section.
We will determine the explicit linear combination below.

See \cite{OWY24} for
an explicit data of basis polynomials 
$\modF{m}{\varepsilon_1, \varepsilon_2}^{(1)}, \ldots,
\modF{m}{\varepsilon_1, \varepsilon_2}^{(d)}$ of
$\eH{m}{\varepsilon_1, \varepsilon_2}$
for higher degrees up to $m\leq 100$.

We recall
\[
f_{3,+} = -y_1y_2y_3,
\qquad
f_{6,-} = \frac{1}{64} (y_1^2-y_2^2)(y_1^2-y_3^2)(y_2^2-y_3^2).
\]

\subsection{The case $\dim \eH{m}{\varepsilon_1, \varepsilon_2}=1$}

\begin{itemize}
\item 
$m=0$:
$\eH{0}{0,0}=\eH{0}{1,0}=\eH{0}{0,1}=\eH{0}{1,1}=
\left< 1 \right>$.

$l=0$ :
$\HG{0,+}=\eH{0}{0,0} 
= \left< 1  \right>$.

$l=3$ : 
$\HG{3,+}=f_{3,+}\times\eH{0}{1,0} 
= f_{3,+}\times\left< 1 \right>
= \left< f_{3,+}  \right>$.

$l=6$ : 
$\HG{6,-}=f_{6,-}\times\eH{0}{0,1} 
=f_{6,-}\times \left< 1 \right>
=\left< f_{6,-} \right> $.

$l=9$ :
$\HG{9,-}= f_{3,+} f_{6,-} \times\eH{0}{1,1}
=f_{3,+} f_{6,-}\times \left< 1 \right>
=\left< f_{3,+} f_{6,-} \right>$.

\[
f_{9,-}=f_{3,+}f_{6,-}.
\]

\item 
$m=4$:
$\eH{4}{\varepsilon_1, \varepsilon_2}
=\left<\modF{4}{\varepsilon_1, \varepsilon_2}\right>$.

$l=4$ : 
$\HG{4,+}=\eH{4}{0,0} =\left< \modF{4}{0,0} \right>$,
\[
\modF{4}{0,0}=\el_1^2-5\el_2,
\qquad
f_{4,+}=-\frac{1}{6} \modF{4}{0,0}.
\]

$l=7$ : 
$\HG{7,+}=f_{3,+}\times \eH{4}{1,0} 
= f_{3,+}\times \left< \modF{4}{1,0} \right>
= \left< f_{3,+} \modF{4}{1,0} \right>$,
\[
\modF{4}{1,0}=\el_1^2-\frac{11}{3}\el_2,
\qquad
f_{7,+}=-\frac{1}{2} f_{3,+} \modF{4}{1,0}.
\]

$l=10$ : 
$\HG{10,-}=f_{6,-}\times \eH{4}{0,1} 
= f_{6,-}\times\left< \modF{4}{0,1} \right>
= \left<f_{6,-}\modF{4}{0,1} \right>$,
\[
\modF{4}{0,1}=\el_1^2-\frac{17}{3}\el_2,
\qquad
f_{10,-}=-\frac{1}{8}f_{6,-}  \modF{4}{0,1}.
\]

$l=13$ : 
$\HG{13,-}= f_{3,+} f_{6,-}\times \eH{4}{1,1} 
= f_{3,+} f_{6,-}\times\left< \modF{4}{1,1} \right>
= \left< f_{3,+}f_{6,-}\modF{4}{1,1} \right>$,
\[
\modF{4}{1,1}=\el_1^2-\frac{23}{5}\el_2,
\qquad
f_{13,-}=-\frac{5}{24}f_{3,+} f_{6,-}  \modF{4}{1,1}.
\]

\item 
$m=6$.

$l=6$ : 
$\HG{6,+}=\eH{6}{0,0} = \left< \modF{6}{0,0} \right>$,
\[
\modF{6}{0,0}=\el_1^3-\frac{21}{2} \el_1\el_2+\frac{231}{2}\el_3,
\qquad
f_{6,+}=\frac{1}{16} \modF{6}{0,0}.
\]

$l=9$ : 
$\HG{9,+}=f_{3,+}\times\eH{6}{1,0}
=f_{3,+} \times \left<  \modF{6}{1,0} \right>
=\left<  f_{3,+} \modF{6}{1,0} \right>$,
\[
\modF{6}{1,0}=\el_1^3-\frac{13}{2}\el_1\el_2+\frac{221}{6}\el_3,
\qquad
f_{9,+}=\frac{3}{16}f_{3,+}\modF{6}{1,0}.
\]

$l=12$ : 
$\HG{12,-} =f_{6,-}\times \eH{6}{0,1}
= f_{6,-}\times\left< \modF{6}{0,1} \right>
= \left< f_{6,-} \modF{6}{0,1} \right>$,
\[
\modF{6}{0,1}=\el_1^3-\frac{19}{2}\el_1\el_2+\frac{437}{2}\el_3,
\qquad 
f_{12,-}=\frac{1}{32}f_{6,-} \modF{6}{0,1}.
\]

$l=15$ : 
$\HG{15,-}= f_{3,+} f_{6,-}\times \eH{6}{1,1} 
= f_{3,+} f_{6,-}\times\left< \modF{6}{1,1} \right>
= \left< f_{3,+}f_{6,-}\modF{6}{1,1} \right>$,
\[
\modF{6}{1,1}=\el_1^3-\frac{15}{2}\el_1\el_2+\frac{145}{2}\el_3,
\qquad 
f_{12,-}=\frac{1}{32}f_{6,-} \modF{6}{1,1}.
\]

\item 
$m=8$.

$l=8$ : 
$\HG{8,+}=\eH{8}{0,0}= \left< \modF{8}{0,0} \right>$,
\[
\modF{8}{0,0}=\el_1^4-18\el_1^2\el_2+65\el_2^2-52\el_1\el_3,
\qquad 
f_{8,+}=\frac{1}{8} \modF{8}{0,0}.
\]

$l=11$ :
$\HG{11,+}=f_{3,+}\times \eH{8}{1,0} 
=f_{3,+}\times  \left< \modF{8}{1,0} \right>
= \left< f_{3,+} \modF{8}{1,0} \right>$,
\[
\modF{8}{1,0} =\el_1^4-10\el_1^2\el_2+\frac{133}{5}\el_2^2-\frac{76}{5}\el_1\el_3,
\qquad
f_{11,+}=\frac{5}{8}f_{3,+} \modF{8}{1,0}.
\]

$l=14$ :
$f_{14,-}= \frac{1}{48} f_{6,-} \modF{8}{0,1}$

$l=17$ : 
$f_{17,-} = \frac{35}{16} f_{3,+} f_{6,-} \modF{8}{1,1}$.

\item 
$m=10$.

$l=10$ : 
$\HG{10,+}=\eH{10}{0,0} = \left< \modF{10}{0,0} \right>$, 
\[
\modF{10}{0,0}=\el_1^5-\frac{55}{2}\el_1^3\el_2+\frac{187}{2}\el_1\el_2^2+\frac{979}{2}\el_1^2\el_3-\frac{3553}{2}\el_2\el_3,
\quad
f_{10,+}=-\frac{1}{384} \modF{10}{0,0}.
\]

$l=13$ :
$f_{13,+}=-\frac{15}{128}f_{3,+} \modF{10}{1,0}$.

$l=16$ :
$f_{16,-}=-\frac{15}{64}f_{6,-} \modF{10}{0,1}$.

$l=19$ : 
$f_{19,-}=-\frac{1}{64}f_{3,+} f_{6,-} \modF{10}{1,1}$.

\item 
$m=14$. omitted.
\end{itemize}

\subsection{The case $\dim \eH{m}{\varepsilon_1, \varepsilon_2}=2$}
\begin{itemize}
\item $m=12$.

$l=12$ : 
$\HG{12,+}  = \left< f_{12,+}^{(1)}, f_{12,+}^{(2)} \right>
= \eH{12}{0,0} 
= \left< \modF{12}{0,0}^{(1)}, \modF{12}{0,0}^{(2)}\right>$,
\[
\modF{12}{0,0}^{(1)}=
\el_1^6-39 \el_1^4 \el_2+\frac{4199}{5} \el_2^3+2652 \el_1^3 \el_3-\frac{75582}{5} \el_1 \el_2 \el_3+\frac{264537}{5} \el_3^2,
\]
\[
\modF{12}{0,0}^{(2)}=
\el_1^2 \el_2^2-\frac{24}{5} \el_2^3-8 \el_1^3 \el_3+\frac{222}{5} \el_1 \el_2 \el_3-\frac{707}{5} \el_3^2.
\]
\[
f_{12,+}^{(1)}=-\frac{7}{1024} \modF{12}{0,0}^{(1)} - \frac{5549}{2048} \modF{12}{0,0}^{(2)},
\qquad
f_{12,+}^{(2)}=\frac{1}{2048} \modF{12}{0,0}^{(1)} - \frac{747}{4096} \modF{12}{0,0}^{(2)}.
\]

$l=15,18,21$ : omitted.

\item $m=16,18,20,22,26$ : omitted.
\end{itemize}

\bibliography{algbib}
\bibliographystyle{amsplain}

\end{document}